\documentclass[11pt,reqno]{amsart}
\usepackage[utf8]{inputenc}
\usepackage{amsmath}
\usepackage{amsfonts}
\usepackage{amssymb}
\usepackage{amsthm}
\usepackage{appendix}
\usepackage{indentfirst}
\usepackage{mathtools}
\mathtoolsset{showonlyrefs=true}
\usepackage{enumitem}
\usepackage{hyperref}
\usepackage{geometry}

\hypersetup{
colorlinks=true,
linkcolor=blue,
citecolor=blue,
}

\title[Multisolitons of 3D Hartree]{Existence of multisoliton solutions of the gravitational Hartree equation in three dimensions}
\author{Yutong Wu}
\address{Department of Mathematics, Yale University, New Haven, CT 06511} 
\email{yutong.wu.yw894@yale.edu}

\def \d{\mathrm{d}}
\def \dt{\mathrm{d}t}
\def \re{\mathrm{Re}}
\def \im{\mathrm{Im}}
\def \G{\mathcal{G}}
\def \E{\mathcal{E}}

\def \H{\mathcal{H}}
\def \M{\mathcal{M}}
\def \P{\mathcal{P}}

\def \N{\mathcal{N}}
\def \J{\mathcal{J}}

\begin{document}

\newtheorem{thm}{Theorem}
\newtheorem{cor}[thm]{Corollary}
\newtheorem*{thm*}{Theorem}
\newtheorem{lem}{Lemma}[section]
\newtheorem{prop}[lem]{Proposition}
\newtheorem{defn}[lem]{Definition}
\newtheorem*{rmk}{Remark}

\numberwithin{equation}{section}

\begin{abstract}
We prove the existence of multisoliton solutions of the three-dimensional gravitational Hartree equation whose trajectories follow many body dynamics of hyperbolic, parabolic or hyperbolic-parabolic types. This work generalizes and improves the result of Krieger-Martel-Rapha\"el \cite{KMR2bodyHartree} on two-soliton solutions.
\end{abstract}

\keywords{Hartree equation, multisoliton, $m$-body problem}

\subjclass[2020]{Primary: 35B40. Secondary: 35Q55}

\maketitle

\section{Introduction} \label{sec intro}

\subsection{Background} 
In 1927, soon after the Schr\"odinger equation was proposed, Douglas Hartree derived the Hartree equation, which provided a way to study many body quantum systems. It has then attracted the interest of both physicists and mathematicians.

In this paper, we consider the gravitational Hartree equation in 3D 
\begin{equation} \label{eq hartree}
    iu_t+ \Delta u - \phi_{|u|^2} u=0,
\end{equation}
where $u: \mathbb{R} \times \mathbb{R}^3 \to \mathbb{C}$ and 
\begin{equation}
    \phi_{|u|^2}= \Delta ^{-1} (|u|^2) = -\frac{1}{4\pi |x|} * |u|^2.
\end{equation}
We begin with some properties of the equation.

The equation possesses a large family of symmetries. Namely, if $u$ solves \eqref{eq hartree}, then for any $(t_0, \alpha_0, \beta_0, \lambda_0, \gamma_0) \in \mathbb{R} \times \mathbb{R}^3 \times \mathbb{R}^3 \times \mathbb{R}_+ \times (\mathbb{R}/2\pi \mathbb{Z})$,
\begin{equation} \label{eq symmetry}
    v(t,x)= \lambda_0^2 u(\lambda_0^2 t+t_0, \lambda_0 x- \alpha_0- \beta_0 t) e^{i (\frac{1}{2} \beta_0 \cdot x- \frac{1}{4}|\beta_0|^2 t+ \gamma_0 )}
\end{equation}
also solves \eqref{eq hartree}. In view of Noether's theorem, we expect the equation to have some conservation laws. The following quantities are conserved by the equation:
\begin{align*}
    &\text{Mass: } && \M(u)= \int |u(t,x)|^2 \d x, \\
    &\text{Momentum: } && \P(u)= \int \im \big( \nabla u(t,x) \overline{u(t,x)} \big) \d x, \\
    &\text{Hamiltonian: } && \H(u)= \frac{1}{2} \int |\nabla u(t,x)|^2 \d x- \frac{1}{4} \int \big| \nabla \phi_{|u|^2} (t,x) \big|^2 \d x,
\end{align*}
In other words, if $u$ solves \eqref{eq hartree}, then these quantities are independent of $t$.

The equation \eqref{eq hartree} is mass ($L^2$)-subcritical. By Theorem 6.1.1 in \cite{cazenave2003semilinear}, we know the Cauchy problem of \eqref{eq hartree} is globally wellposed in $H^1$. To be more precise, for any $u_0 \in H^1(\mathbb{R}^3)$, there exists a unique $u \in C \big( \mathbb{R}; H^1(\mathbb{R}^3) \big)$ satisfying \eqref{eq hartree} and $u(0,x)=u_0(x)$. Moreover, this $u$ depends on $u_0$ continuously.

The linear Schr\"odinger equation $iu_t+ \Delta u=0$ is called dispersive because of the dispersive estimate $\Vert e^{it \Delta} f \Vert_{L^\infty} \lesssim t^{-\frac{3}{2}} \Vert f \Vert_{L^1}$ for any $f \in L^1 \cap L^2 (\mathbb{R}^3),\ t>0$. However, due to the (focusing) nonlinearity, there exist non-dispersive solutions to \eqref{eq hartree} called solitary waves. 

A solitary wave is a solution to \eqref{eq hartree} of the form $u(t,x)= e^{it} W(x)$. We deduce that $W$ satisfies $\Delta W- \phi_{|W|^2} W= W$. From \cite{Liebgroundstate} we know there exists a unique radial and nonnegative solution $Q$ of
\begin{equation} \label{eq ground state}
    \Delta Q- \phi_{Q^2} Q= Q,
\end{equation}
called the ground state. It is proved in \cite{Lionsminizingdescriptionofgroundstate} that this $Q$ can also be characterized as the radial minimizer of the Hamiltonian subject to a given $L^2$ norm. More precisely, it minimizes
\begin{equation}
    \H(u)= \frac{1}{2} \int |\nabla u|^2- \frac{1}{4} \int |\nabla \phi_{|u|^2} |^2, 
\end{equation}
among all $u \in H^1$ with the same $L^2$ norm as $Q$. Another property of $Q$ is the exponential decay:
\begin{equation} \label{eq decay of Q}
    Q(x) \le C e^{-c|x|}, \quad \forall x \in \mathbb{R}^3.
\end{equation}Using \eqref{eq symmetry} we can construct a family of ground state solitary waves. 

The main object of interest in this paper is the multisolitary wave, or multisoliton, which can be roughly understood as the sum of several solitary waves. 

Multisolitary waves are expected to be important components of generic solutions as claimed in the soliton resolution conjecture. It plays an crucial role when we try to understand the long time behavior of solutions. We cite \cite{Taosoliton}, \cite{solitonresolutionNLS}, \cite{TaoNLS}, \cite{JendreyLawrieNLW}, \cite{DKMwaveodd} and \cite{solitonresolutionmKdV} as references for some partial results on soliton resolution for some nonlinear dispersive equations. Another aspect is to study the existence and stability of solutions that approach given multisolitons. In this direction, \cite{KsolitarywavesofNLS}, \cite{MartelgKdV} and \cite{NsolitonKleinGordon} proved the existence of solutions that asymptotically approach multisoliton with constant and distinct speeds for the nonlinear Schr\"odinger equation (NLS), the generalized Korteweg-de Vries equation (gKdV) and the nonlinear Klein Gordon equation (NLKG), respectively. Moreover, \cite{StabilityNLS1d} and \cite{MMTstabilityforgKdV} studied the stability of such multisoliton solutions. The general expectation is that multisolitons are orbitally stable if the speeds are separated.

We point out that the above literature only considered equations with local nonlinearity. For such equations, the sum of two ground state solitary waves moving away at a constant speed solves the equation up to a term that decays exponentially in time. This reflects that the nonlinearity does not affect the asymptotic behavior dramatically. On the other hand, the long time behavior of the Hartree equation is difficult to study because of the long range effect of the nonlinearity. More precisely, we have
\begin{equation}
    \phi_{Q^2}(x) \sim \frac{1}{|x|} \quad \text{as } x \to \infty,
\end{equation}
where $\sim$ means comparable up to constants, so the error term at most admits a polynomial decay. A quantitative estimate of such errors is given in Lemma \ref{lem localization}. This is the main difficulty we have to deal with.

\subsection{The m-body problem}

As a starting point of the study of long time dynamics, Krieger, Martel and Rapha\"el \cite{KMR2bodyHartree} studied the existence of two-soliton solutions of \eqref{eq hartree}. This pioneer paper revealed that one should expect a gravitational two-body interaction within the two solitons. This interaction cancels the long range effect properly. In this subsection, we review some important facts about the $m$-body problem.

Let $m \ge 2$. The $\mathbf{m}$\textbf{-body problem} is an ODE system
\begin{equation} \label{eq m-body problem} 
    \dot{\alpha}_j(t)= 2\beta_j(t), \quad \dot{\beta}_j(t)= - \sum_{k \neq j} \frac{\Vert Q \Vert_{L^2}^2}{4\pi \lambda_k} \cdot \frac{\alpha_j(t)- \alpha_k(t)} {\left| \alpha_j(t)- \alpha_k(t) \right|^3}, \qquad \forall 1 \le j \le m,
\end{equation}
where $\alpha_j, \beta_j \in C^1 \big( \mathbb{R}, \mathbb{R}^3 \big)$ and $\lambda_j \in \mathbb{R}_+$ for $1 \le j \le m$. 

The case where $m=2$ is the famous two-body problem. It is known that the solution to the two-body problem is hyperbolic, parabolic or elliptic, corresponding to the asymptotic behavior being
\begin{equation}
    |\alpha_1(t)- \alpha_2(t)| \sim t^q \text{ as } t \to +\infty
\end{equation}
with $q=1$, $q=\frac{2}{3}$ or $q=0$, respectively. In this paper, we write $f \sim g$ if there exist $0<c<C$ such that $cf \le g \le Cf$.

The $m$-body problem for $m \ge 3$, as opposed to the two-body problem, is much more complicated. Even for three-body, chaotic dynamics may occur \cite{Poincare3body}. For our purpose, we will focus on \textbf{expansive} solutions, which means $|\alpha_j(t)- \alpha_k(t)| \to +\infty$ as $t \to +\infty$ for any $j \neq k$. On the PDE side, this requires the centers of solitons to move far away from each other, which is essential for us to exploit the exponential decay of the ground state \eqref{eq decay of Q}.

By translation we may set the center of the system at the origin. Consider the following sets of configurations
\begin{gather*}
    \mathcal{X}= \Big\{ (x_1, \cdots, x_m) \in \mathbb{R}^{3m} \ \Big| \ \sum_{j=1}^m \lambda_j^{-1} x_j=0 \Big\}, \\
    \mathcal{Y}= \big\{ (x_1, \cdots, x_m) \in \mathcal{X} \ \big| \ x_j \neq x_k, \ \forall j \neq k \big\} \quad \text{and} \quad \Delta= \mathcal{X} \setminus  \mathcal{Y}.
\end{gather*}

The final evolution of expansive solutions is described as follows.
\begin{thm*} [Marchal-Saari \cite{StructureOfExpansive}]
For an expansive solution of \eqref{eq m-body problem} centered at the origin, there exists $(a_1, \cdots, a_m) \in \mathcal{X}$ such that 
\begin{equation} \label{eq expansive solution}
    \alpha_j= a_jt+ O(t^\frac{2}{3}),\ 1 \le j \le m \quad \text{as } t \to +\infty.
\end{equation}
Moreover, if $a_j=a_k$ for some $j \neq k$,  then $|\alpha_j- \alpha_k| \sim t^\frac{2}{3}$ as $t \to +\infty$.
\end{thm*}

We classify solutions of the form \eqref{eq expansive solution} into the following three types in the spirit of Chazy \cite{ChazyClassification} according to the growth of distances between different bodies.
\begin{itemize}
    \item \textbf{hyperbolic}: $(a_1, \cdots, a_m) \in \mathcal{Y}$, or equivalently, for all $j \neq k$, 
    \begin{equation}
        |\alpha_j(t)- \alpha_k(t)| \sim t \quad \text{as } t \to +\infty.
    \end{equation}
    \item \textbf{parabolic}: $(a_1, \cdots, a_m)=0$, or equivalently, for all $j \neq k$, 
    \begin{equation}
        |\alpha_j(t)- \alpha_k(t)| \sim t^\frac{2}{3} \quad \text{as } t \to +\infty.
    \end{equation}
    \item \textbf{hyperbolic-parabolic}: $(a_1, \cdots, a_m) \in \Delta \setminus \{0\}$. In this case, both of the above pairwise asymptotics occur, and they are the only possibilities.
\end{itemize}

Notice that this agrees with the previous notion for two-body dynamics, where the hyperbolic-parabolic dynamic does not appear evidently.

The existence of such three types of solutions has been studied in the recent years. The newest results in this direction are obtained through a variational approach originating from Maderna-Venturelli  \cite{nbodyhyperbolic}. See also \cite{Nbodyparabolic} for a similar strategy. We state the results as follows.

\begin{thm*} [Maderna-Venturelli \cite{nbodyhyperbolic}; Polimeni-Terracini \cite{ExistenceofNbodyproblem}] Let $\lambda_j \in \mathbb{R}_+$ for $1 \le j \le m$.

\begin{enumerate} [label=(\arabic*)]
    \item There exists a hyperbolic solution to \eqref{eq m-body problem} of the form
    \begin{equation} \label{eq hyperbolic solution}
        \alpha_j(t)= a_j t+ O(\log t) \quad \text{as } t \to +\infty
    \end{equation}
    for any $(a_1, \cdots, a_m) \in \mathcal{Y}$ and initial configuration in $\mathcal{X}$.
    \item There exists a parabolic solution to \eqref{eq m-body problem} of the form
    \begin{equation} \label{eq parabolic solution}
        \alpha_j(t)= cb_j t^{\frac{2}{3}}+ o(t^{\frac{1}{3}+}) \quad \text{as } t \to +\infty
    \end{equation}
    for any minimal $(b_1, \cdots, b_m) \in \mathcal{Y}$ and initial configuration in $\mathcal{X}$, where $c>0$ is determined by $b_1,\cdots,b_m$.
    \item There exists a hyperbolic-parabolic solution to \eqref{eq m-body problem} of the form
    \begin{equation} \label{eq hyperbolic-parabolic solution}
        \alpha_j(t)=a_j t+ c_jb_jt^{\frac{2}{3}}+ o(t^{\frac{1}{3}+}) \quad \text{as } t \to +\infty
    \end{equation}
    for any $(a_1, \cdots, a_m) \in \Delta \setminus \{0\}$, minimal $(b_1, \cdots, b_m) \in \mathcal{Y}$ and initial configuration in $\mathcal{X}$, where $c_j>0$ is determined by $a_1,\cdots,a_m, b_1,\cdots,b_m$ and $c_j=c_k$ whenever $a_j=a_k$.
\end{enumerate}
\end{thm*}

Regarding the term minimal, see Remark 2 after Theorem \ref{thm existence}. 

\subsection{The main result}

The result in \cite{KMR2bodyHartree} is that for the two-body problem, hyperbolic and parabolic solutions to \eqref{eq m-body problem} produce two-soliton solutions of \eqref{eq hartree}. Based on their method, we generalize their result to $m$-soliton solutions. Our result asserts the existence of multisoliton solutions to \eqref{eq hartree} reproducing the above three expansive dynamics. An assumption on the masses is needed for the last two dynamics.

\begin{thm} \label{thm existence}
Let $(\alpha_1^\infty, \cdots, \alpha_m^\infty, \beta_1^\infty, \cdots, \beta_m^\infty, \lambda_1^\infty, \cdots, \lambda_m^\infty)$ be a solution to \eqref{eq m-body problem} of one of the three types \eqref{eq hyperbolic solution}, \eqref{eq parabolic solution} or \eqref{eq hyperbolic-parabolic solution}. Suppose $\lambda_j^\infty= \lambda_k^\infty$ whenever $|\alpha_j^\infty(t)- \alpha_k^\infty(t)| \sim t^{\frac{2}{3}}$ as $t \to +\infty$.

Then there exists a solution $u$ to \eqref{eq hartree} and $\gamma_1^\infty(t), \cdots \gamma_m^\infty(t)$ that are $C^1$ in $t$ such that
\begin{equation}
    \lim_{t \to +\infty} \bigg\Vert u(t,x)- \sum_{j=1}^m \frac{1}{(\lambda_j^\infty)^2} Q \Big( \frac{x - \alpha_j^\infty(t)}{\lambda_j^\infty} \Big) e^{-i\gamma_j^\infty(t)+ i\beta_j^\infty(t) \cdot x} \bigg\Vert_{H^1} =0.
\end{equation}
\end{thm}

\begin{rmk} \ 

1. In the parabolic case, Theorem \ref{thm existence} improves the result in \cite{KMR2bodyHartree} as we take $\alpha_j$ (in their statement) to be identical to $\alpha_j^\infty$, which trivially answers their Comment 2. 

2. The assumption that $(b_1, \cdots, b_m)$ is minimal in \eqref{eq parabolic solution} and \eqref{eq hyperbolic-parabolic solution} is not directly used when we deal with the parabolic case and the hyperbolic-parabolic case. The precise definition of ``minimal" can be found in \cite{ExistenceofNbodyproblem}, and this assumption is needed there to guarantee the existence of solutions of the $m$-body problem. Note also that this refers to different properties in \eqref{eq parabolic solution} and \eqref{eq hyperbolic-parabolic solution}.

3. Our result seems a satisfactory counterpart of \cite{KsolitarywavesofNLS} and \cite{MartelgKdV}, which dealt with the existence of multisolitary waves of (NLS) and (gKdV), respectively. Moreover, by \cite{nonscatterhartree}, a radiation term is not expected for multisoliton solutions of the Hartree equation. Thus in the spirit of the soliton resolution conjecture, we have constructed a relatively complete class of solutions.

4. Some future problems related to this work include the following: Are the multisoliton solutions constructed above stable? Do multisoliton solutions with elliptic type interactions exist? Can the results be extended to other dimensions?
\end{rmk}

Using the expansion of hyperbolic motions given by Chazy \cite{ChazyClassification}, for any $(x_1, \cdots, x_m) \in \mathcal{X}$ and $(a_1, \cdots, a_m) \in \mathcal{Y}$, there exists a solution to \eqref{eq m-body problem} of the form $\alpha_j(t)= x_j+ a_j t+ c_j \log t+ o(1)$ for some $(c_1, \cdots, c_m) \in \mathcal{X}$. Then we deduce the following from Theorem \ref{thm existence}.

\begin{cor} \label{cor hyperbolic}
Given $\lambda_1, \cdots, \lambda_m \in \mathbb{R}_+$, $x_1, \cdots, x_m \in \mathbb{R}^3$ and distinct $a_1, \cdots, a_m \in \mathbb{R}^3$, there exist $c_1, \cdots, c_m \in \mathbb{R}^3$, a solution $u$ to \eqref{eq hartree}, and $\gamma_1(t), \cdots \gamma_m(t)$ that are $C^1$ in $t$ such that
\begin{equation}
    \lim_{t \to +\infty} \bigg\Vert u(t,x)- \sum_{j=1}^m \frac{1}{\lambda_j^2} Q \Big( \frac{x -x_j- a_jt- c_j \log t}{\lambda_j} \Big) e^{-i\gamma_j(t)+ i\frac{a_j}{2} \cdot x} \bigg\Vert_{H^1} =0.
\end{equation}
\end{cor}

\begin{rmk}
Comparing Corollary \ref{cor hyperbolic} with the results in \cite{KsolitarywavesofNLS} and \cite{MartelgKdV}, we see that the $c_j \log t$ term is the necessary corrector accounting for the long range effect of the Hartree nonlinearity. 
\end{rmk}

We end the introduction section with some comments on the proof of the theorem and the organization of the paper.

Due to the long range effect mentioned before, we need to first construct approximate solutions. The difficulty compared to \cite{KMR2bodyHartree} lies mainly in the parabolic and hyperbolic-parabolic cases. We need to study an approximate system of the $m$-body problem, which is essentially harder than the two-body case. For this purpose, we have to perform delicate computation of the constants involved. We made use of a cancellation of errors displayed in the proof of Proposition \ref{prop parabolic and hyperbolic-parabolic trajectory}. This is a new observation.

The article is organized as follows. In Section \ref{sec approximate}, we construct approximate multisolitary solutions of \eqref{eq hartree} up to the $N$-th order for any $N \ge 1$ to overcome the long range effect. We then focus on the hyperbolic case. We reduce the problem to a uniform estimate and furthermore a modulation estimate in Section \ref{sec reduction}. Then the modulation estimate is proved in Section \ref{sec estimate}, finishing the proof of the hyperbolic case. The other two cases of Theorem \ref{thm existence} are addressed in section \ref{sec parabolic and hyperbolic-parabolic}.

\section{Approximate solutions} \label{sec approximate}

First we introduce some notations. For $\alpha_j$, $\beta_j$, $\lambda_j$ and $\gamma_j$ (may depending on time), we denote
\begin{equation} \label{eq notation, gathered positions} \begin{gathered}
    \alpha= (\alpha_1, \cdots, \alpha_m), \quad \beta= (\beta_1, \cdots, \beta_m), \\
    \lambda= (\lambda_1, \cdots, \lambda_m), \quad \gamma= (\gamma_1, \cdots, \gamma_m), \\
    P= (\alpha, \beta, \lambda), \quad g=(P,\gamma), \quad g_j= (\alpha_j, \beta_j, \lambda_j, \gamma_j), \\
    \alpha_{jk}= \alpha_j- \alpha_k, \quad \beta_{jk}= \beta_j- \beta_k, \quad a= \min_{j \neq k} |\alpha_{jk}|.
\end{gathered} \end{equation}
We use similar notation when there are superscripts.

For $u: \mathbb{R} \times \mathbb{R}^3 \to \mathbb{C}$, we define $g_j u: \mathbb{R} \times \mathbb{R}^3 \to \mathbb{C}$ by
\begin{equation} \label{eq notation, action} \begin{gathered}
    g_j u(t,x)= \frac{1}{\lambda_j^2} u \Big( t, \frac{x-\alpha_j}{\lambda_j} \Big) e^{-i\gamma_j+ i\beta_j \cdot x}.
\end{gathered} \end{equation}
In particular, $g_j Q$ represents a soliton and $\sum \limits_{j=1}^m g_j Q$ is a multisoliton.

Assume $u=g_j v$ and the components of $g_j$ may depend on $t$. More precisely, we assume
\begin{equation}
    u(t,x)= \frac{1}{\lambda_j^2(t)} v \left( t, \frac{x-\alpha_j(t)}{\lambda_j(t)} \right) e^{-i \gamma_j(t)+ i \beta_j(t) \cdot x}.
\end{equation}
Then we have
\begin{equation} \label{eq basic calculation} \begin{aligned}
    &u_t= \frac{1}{\lambda_j^4} \Big( \lambda_j^2 v_t- \lambda_j \dot{\alpha}_j \cdot \nabla v- \dot{\lambda}_j (x-\alpha_j) \cdot \nabla v- 2\lambda_j \dot{\lambda}_j v \\
    &\qquad \qquad \qquad \qquad \qquad \qquad \qquad -i\lambda_j^2 \dot{\gamma}_j v+ i\lambda_j^2 \dot{\beta}_j \cdot x v \Big) e^{-i\gamma_j+ i\beta_j \cdot x}, \\
    &\nabla u= \frac{1}{\lambda_j^3} \left( \nabla v+ i\lambda_j \beta_j v \right) e^{-i\gamma_j+ i\beta_j \cdot x}, \\
    &\Delta u= \frac{1}{\lambda_j^4} \left( \Delta v+ 2i\lambda_j \beta_j \cdot \nabla v- \lambda_j^2 |\beta_j|^2 v \right) e^{-i\gamma_j+ i\beta_j \cdot x}, \\
    &\phi_{|u|^2}(x)= \frac{1}{\lambda_j^2} \phi_{|v|^2} \Big( \frac{x-\alpha_j}{\lambda_j} \Big).
\end{aligned} \end{equation}

Therefore, if we let
\begin{equation}
    u(t,x):= \sum_{j=1}^m u_j(t,x) := \sum_{j=1}^m g_j v_j(t,x),
\end{equation}
and set $y_j= \frac{x-\alpha_j(t)}{\lambda_j(t)}$, $\Lambda v_j= 2v_j+ y_j \cdot \nabla v_j$, then
\begin{equation}
    iu_t+ \Delta u -\phi_{|u^2|}u = \sum_{j=1}^m \frac{1}{\lambda_j^4} E_j(t,y_j) e^{-i\gamma_j+ i\beta_j \cdot x}- \sum_{k \neq j} \phi_{\mathrm{Re} (u_k \overline{u_j})} u,
\end{equation}
where
\begin{equation} \begin{aligned}
    E_j(t,y_j)= &\quad i\lambda_j^2 \partial_t v_j+ \Delta v_j- v_j- i\lambda_j \dot{\lambda}_j \Lambda v_j- \lambda_j^3 \dot{\beta}_j \cdot y_j v_j \\
    &- i\lambda_j \left( \dot{\alpha}_j- 2\beta_j \right) \nabla v_j+ \lambda_j^2 \left( \dot{\gamma}_j+ \frac{1}{\lambda_j^2}- |\beta_j|^2- \dot{\beta}_j \cdot \alpha_j \right) v_j \\
    &- \left[ \phi_{|v_j|^2}+ \sum_{k \neq j} \left( \frac{\lambda_j}{\lambda_k} \right)^2 \phi_{|v_k|^2} \left( t, \frac{\lambda_j}{\lambda_k} y_j+ \frac{\alpha_{jk}}{\lambda_k} \right) \right] v_j.
\end{aligned} \end{equation}
To be clear, the space variable of the right hand side is $y_j$ unless explicitly written out.

\subsection{Definition of approximate solutions}

We need to approximate the last term of $E_j$. Since
\begin{equation}
    \left( \frac{\lambda_j}{\lambda_k} \right)^2 \phi_{|v_k|^2} \left( t, \frac{\lambda_j}{\lambda_k} y_j+ \frac{\alpha_{jk}}{\lambda_k} \right) = -\frac{\lambda_j^2}{4\pi \lambda_k} \int_{\mathbb{R}^3} \frac{|v_k(t,\xi)|^2}{|\lambda_j y_j+ \alpha_{jk}- \lambda_k \xi|} \d \xi,
\end{equation}
we consider the Taylor expansion
\begin{equation}
    \frac{1}{|\alpha-\zeta|}= \sum_{n=1}^N F_n(\alpha,\zeta)+ O \left( \frac{|\zeta|^N}{|\alpha|^{N+1}} \right) \quad \text{as } \zeta \to 0,
\end{equation}
where $F_n(\alpha,\zeta)$ is homogeneous of degree $-n$ in $\alpha$ and of degree $n-1$ in $\zeta$. We define the approximation to be
\begin{equation} \begin{aligned}
    \phi_{|v_k|^2}^{(N)}(t,y_j) &:= \sum_{n=1}^N \psi_{|v_k|^2}^{(n)}(t,y_j) \\
    &:= \sum_{n=1}^N -\frac{\lambda_j^2}{4\pi \lambda_k} \int_{\mathbb{R}^3} |v_k(t,\xi)|^2 F_n(\alpha_{jk}, \lambda_k \xi- \lambda_j y_j) \d \xi.
\end{aligned} \end{equation}
Explicit formulae for the first few terms are as follows:
\begin{gather*}
    \psi_{|v_k|^2}^{(1)}(t,y_j)= -\frac{\lambda_j^2}{4\pi \lambda_k |\alpha_{jk}|} \int |v_k(t,\xi)|^2 \d \xi, \\
    \psi_{|v_k|^2}^{(2)}(t,y_j)= -\frac{\lambda_j^2}{4\pi \lambda_k |\alpha_{jk}|^3} \int \Big( \lambda_k (\alpha_{jk} \cdot \xi) |v_k(t,\xi)|^2- \lambda_j (y_j \cdot \alpha_{jk}) |v_k(t, \xi)|^2 \Big) \d \xi.
\end{gather*}
We will need to take $v_k=Q$. In this case we denote $\psi_{|v_k|^2}^{(n)}$ by $\psi_{Q^2,k}^{(n)}$. Namely, 
\begin{equation}
    \psi_{Q^2,k}^{(n)}(y_j)= -\frac{\lambda_j^2}{4\pi \lambda_k} \int_{\mathbb{R}^3} Q^2(\xi) F_n(\alpha_{jk}, \lambda_k \xi- \lambda_j y_j) \d \xi.
\end{equation}

We shall let $v_j$ vary in $N$, and we also assume $v_j$ depends on time only through the parameters $t \mapsto P(t)$, which means $v_j(t,y_j)= V_j^{(N)}(P(t),y_j)$ for some $V_j^{(N)}$. 

Define
\begin{equation} \label{eq approximate solution} \begin{aligned}
    R_g^{(N)}(t,x):= \sum_{j=1}^m R_{j,g}^{(N)}(t,x):= \sum_{j=1}^m g_j V_j^{(N)}(P(t),x). 
\end{aligned} \end{equation}
Let us omit the subscript $g$ of $R^{(N)}$ for now. We have
\begin{equation} \label{eq equation of approximate solution} \begin{aligned}
    &i\partial_t R^{(N)}+ \Delta R^{(N)}- \phi_{|R^{(N)}|^2} R^{(N)} \\
    = &\quad \sum_{j=1}^m \frac{1}{\lambda_j^4} E_j^{(N)}(t,y_j) e^{-i\gamma_j+ i\beta_j \cdot x}- \sum_{k \neq j} \phi_{\mathrm{Re} (R_k^{(N)} \overline{R_j^{(N)}})} R^{(N)} \\
    &+ \sum_{j=1}^m \frac{1}{\lambda_j^4} \sum_{k \neq j} \left[ \phi_{\left| V_k^{(N)} \right|^2}^{(N)}- \left( \frac{\lambda_j}{\lambda_k} \right)^2 \phi_{|V_k^{(N)}|^2} \left( P(t), \frac{\lambda_j}{\lambda_k} y_j+ \frac{\alpha_{jk}}{\lambda_k} \right) \right] V_j^{(N)} e^{-i\gamma_j+ i\beta_j \cdot x},
\end{aligned} \end{equation} 
where
\begin{equation} \begin{aligned}
    E_j^{(N)}= &\quad i\lambda_j^2 \partial_t V_j^{(N)}+ \Delta V_j^{(N)}- V_j^{(N)}- i\lambda_j \dot{\lambda}_j \Lambda V_j^{(N)}- \lambda_j^3 \dot{\beta}_j \cdot y_j V_j^{(N)} \\
    &- i\lambda_j \left( \dot{\alpha}_j- 2\beta_j \right) \nabla V_j^{(N)}+ \lambda_j^2 \left( \dot{\gamma}_j+ \frac{1}{\lambda_j^2}- |\beta_j|^2- \dot{\beta}_j \cdot \alpha_j \right) V_j^{(N)} \\
    &- \left( \phi_{\left| V_j^{(N)} \right|^2}+ \sum_{k \neq j} \phi_{\left| V_k^{(N)} \right|^2}^{(N)} \right) V_j^{(N)}.
\end{aligned} \end{equation}

For functions $M_j^{(N)}(P)$ and $B_j^{(N)}(P)$ of the parameters, we can decompose
\begin{equation}
    E_j^{(N)}= \tilde{E}_j^{(N)}+ S_j^{(N)},
\end{equation}
where
\begin{equation} \label{eq definition of E_j tilde} \begin{aligned} 
    \tilde{E}_j^{(N)}(t,y_j)= &\quad \Delta V_j^{(N)}- V_j^{(N)}- \phi_{\left| V_j^{(N)} \right|^2} V_j^{(N)}- \sum_{k \neq j} \phi_{\left| V_k^{(N)} \right|^2}^{(N)} V_j^{(N)} \\
    &- i\lambda_j M_j^{(N)} \Lambda V_j^{(N)}- \lambda_j^3 B_j^{(N)} \cdot y_j V_j^{(N)} \\
    &+ i\lambda_j^2 \sum_{k=1}^m \bigg( \frac{\partial V_j^{(N)}}{\partial \alpha_k} \cdot 2\beta_k+ \frac{\partial V_j^{(N)}}{\partial \beta_k} \cdot B_k^{(N)}+ \frac{\partial V_j^{(N)}}{\partial \lambda_k} M_k^{(N)} \bigg)  
\end{aligned} \end{equation} 
and
\begin{equation} \label{eq definition of S_j^N} \begin{aligned}
    S_j^{(N)}(t,x)= &- i\lambda_j \left( \dot{\alpha}_j- 2\beta_j \right) \nabla V_j^{(N)}- \lambda_j^3 \left( \dot{\beta}_j- B_j^{(N)} \right) \cdot y_j V_j^{(N)} \\
    &- i\lambda_j \left( \dot{\lambda}_j - M_j^{(N)} \right) \Lambda V_j^{(N)}+ \lambda_j^2 \bigg( \dot{\gamma}_j+ \frac{1}{\lambda_j^2}- |\beta_j|^2- \dot{\beta}_j \cdot \alpha_j \bigg) V_j^{(N)} \\
    &+ i\lambda_j^2 \sum_{k=1}^m \Bigg[ \frac{\partial V_j^{(N)}}{\partial \alpha_k} \cdot \left( \dot{\alpha}_k- 2\beta_k \right)+ \frac{\partial V_j^{(N)}}{\partial \beta_k} \cdot \left( \dot{\beta}_k- B_k^{(N)} \right) \\
    &\qquad \qquad \qquad \qquad \qquad \qquad \qquad \qquad + \frac{\partial V_j^{(N)}}{\partial \lambda_k} \left( \dot{\lambda}_k- M_k^{(N)} \right) \Bigg].
\end{aligned} \end{equation}

Note that $\tilde{E}_j^{(N)}$ is set to be a function of $y_j$ instead of $x$. This is to align with a later statement. It does not matter whether $S_j^{(N)}$ is a function of $y_j$ or $x$, but we will let it be a function of $x$ for preciseness.

We remark that $\tilde{E}_j^{(N)}$ accounts for the error terms arising from the nonlinearity, while $S_j^{(N)}$ contains the error terms caused by the parameters. We introduce $M_j^{(N)}$ and $B_j^{(N)}$ to provide flexibility in controlling $\tilde{E}_j^{(N)}$.

Next, we show that we can choose $V_j^{(N)}$, $M_j^{(N)}$, and $B_j^{(N)}$ so that $\tilde{E}_j^{(N)}$ is small. This smallness will be a result of homogeneity, so we give the following definition.

\begin{defn}[\textbf{Admissible functions}] \ \par
Recalling \eqref{eq notation, gathered positions}, let $\Omega$ denote the space of non-collision positions:
\begin{equation} \begin{aligned}
    \Omega:= \Big\{ P= (\alpha, \beta, \lambda) \in \mathbb{R}^{3m} \times \mathbb{R}^{3m} \times \mathbb{R}_+^m \ \big| \ \alpha_j \neq \alpha_k,\ \forall j \neq k \Big\}.
\end{aligned} \end{equation}

(1) Let $n \in \mathbb{N}$. Define $S_n$ to be the set of functions $\sigma: \Omega \to \mathbb{C}$ that is homogeneous in $\alpha$ of degree $-n$ and is a finite sum of 
\begin{equation}
    c\prod_{j \neq k} (\alpha_j- \alpha_k)^{p_{jk}} |\alpha_j- \alpha_k|^{-q_{jk}} \prod_{j=1}^m \beta_j^{k_j} \lambda_j^{l_j},
\end{equation}
where $c \in \mathbb{C}$, $p_{jk} \in \mathbb{N}^3$, $q_{jk} \in \mathbb{N}$, $k_j \in \mathbb{N}^3$, $l_j \in \mathbb{Z}$ and $|p_{jk}| \le q_{jk}$.

(2) We say a function $u: \Omega \times \mathbb{R}^3 \to \mathbb{C}$ is \textbf{admissible} if $u$ is a finite sum of 
\begin{equation}
    \sigma(\alpha_1, \cdots, \alpha_m, \beta_1, \cdots, \beta_m, \lambda_1, \cdots \lambda_m) \tau(x),
\end{equation}
where $\sigma \in S_n$ for some $n \in \mathbb{N}$ and $\tau \in C^\infty$ satisfies
\begin{equation}
    \big| \nabla^k \tau(x) \big| \le e^{-c_k|x|}, \qquad \forall k \ge 0,\ x \in \mathbb{R}^3.
\end{equation}

If $n$ is the same for all addends, then we say $u$ is admissible of degree $n$. Otherwise, taking $n$ as the minimal one among all addends, we say $u$ is admissible of degree $\ge n$.
\end{defn}

Here are some properties of admissible functions.
\begin{lem} \label{lem properties of admissible functions}
Let $n,m \in \mathbb{N}$ and $u,v$ be admissible of degree $n,m$, respectively. Then
\begin{enumerate} [label=(\arabic*)]
    \item $\forall j$, $\frac{\partial u}{\partial \alpha_j}$ is admissible of degree $1+n$, and $\frac{\partial u}{\partial \beta_j}, \frac{\partial u}{\partial \lambda_j}$ are admissible of degree $n$.
    \item $uv$ is admissible of degree $n+m$;
    \item $\phi_u v$ is admissible of degree $n+m$;
    \item $\forall k \ge 1$, $\psi_u^{(k)} v$ is admissible of degree $k+n+m$;
    \item $\forall N \ge 1$, $\phi_u^{(N)} v$ is admissible of degree $\ge 1+n+m$;
    \item $\exists c>0$ such that for any compact set $K \subset \mathbb{R}^{3m} \times \mathbb{R}_+^m$, $\exists C_K>0$ such that
    \begin{equation} \label{eq estimate of admissible functions}
        |u(\alpha, \beta, \gamma, x)| \le \frac{C_K}{a^n} e^{-c|x|}, \quad \forall (\beta, \gamma) \in K,
    \end{equation}
    where $a= \min \limits_{j \neq k} |\alpha_j- \alpha_k|$ as in \eqref{eq notation, gathered positions}.
\end{enumerate}
\end{lem}

The proof of these properties is direct so we shall omit it. The point of considering admissible functions is that according to (6), they decay rapidly when $n$ is large.

Consider the linearized operators $L_+, L_-$ around $Q$ defined by
\begin{equation}
    L_+f:= -\Delta f+ f+ \phi_{Q^2}f+ 2\phi_{Qf}Q, \quad L_-f:= -\Delta f+ f+ \phi_{Q^2}f.
\end{equation}
By Theorem 4 in \cite{Lenzmanngroundstate}, $\{ \partial_1 Q, \partial_2 Q, \partial_3 Q\}$ spans $\ker L_+$, and $\{Q\}$ spans $\ker L_-$. Moreover, Lemma 2.4 in \cite{KMR2bodyHartree} asserts that when restricted to admissible functions, $\ker (L_{\pm})^\perp$ is exactly the range of $L_{\pm}$. A precise statement is as follows.

\begin{lem} \label{lem inverse of L+, L-}
Let $n \in \mathbb{N}$ and $f$ be real-valued and admissible of degree $n$. 
\begin{enumerate} [label=(\arabic*)]
    \item If $(f, \nabla Q)=0$, then $L_+u=f$ has a real-valued solution $u$ admissible of degree $n$. 
    \item If $(f, Q)=0$, then $L_-u=f$ has a real-valued solution $u$ admissible of degree $n$. 
\end{enumerate}
Furthermore, if $f$ is radial, then $u$ can be chosen to be radial.
\end{lem}

The following proposition constructs the approximate solutions.

\begin{prop} \label{prop construction of approximate bubbles}
For $n \ge 1$ and $1 \le j \le m$, there exist real-valued $m_j^{(n)}, b_j^{(n)} \in S_n$ and $T_j^{(n)}$ that is admissible of degree $n$ such that: for any $N \ge 1$, if setting 
\begin{gather*}
    V_j^{(N)}(P,y_j)= Q(y_j)+ \sum_{n=1}^N T_j^{(n)}(P,y_j), \\
    M_j^{(N)}(P)= \sum_{n=1}^N m_j^{(n)}(P) \quad \text{and} \quad B_j^{(N)}(P)= \sum_{n=1}^N b_j^{(n)}(P),    
\end{gather*}
then $\tilde{E}_j^{(N)}$ defined by \eqref{eq definition of E_j tilde} is admissible of degree $\ge N+1$.
\end{prop}

\begin{proof}
We construct the functions by induction in $N$. 

For $N=1$, we take $m_j^{(1)}= b_j^{(1)}=0$. Suppose $T_j^{(1)}$ is admissible of degree $1$ and real-valued. By Lemma \ref{lem properties of admissible functions} and \eqref{eq ground state}, we have
\begin{equation} \begin{aligned}
    \tilde{E}_j^{(1)} &= \Delta V_j^{(1)}- V_j^{(1)}+ i\lambda_j^2 \sum_{k=1}^m \frac{\partial V_j^{(1)}}{\partial \alpha_k} \cdot 2\beta_k - \bigg( \phi_{|V_j^{(1)}|^2}+ \sum_{k \neq j} \psi_{|V_k^{(1)}|^2}^{(1)} \bigg) V_j^{(1)} \\
    &= \Delta T_j^{(1)}- T_j^{(1)}- \phi_{Q^2} T_j^{(1)}- 2\phi_{Q T_j^{(1)}} Q- \sum_{k \neq j} \psi_{Q^2,k}^{(1)} Q+ error \\
    &= -L_+ T_j^{(1)}+ \sum_{k \neq j} \frac{\lambda_j^2 \Vert Q \Vert _{L^2}^2} {4\pi \lambda_k |\alpha_{jk}|} Q+ error,
\end{aligned} \end{equation}
where $error$ is admissible of degree $\ge 2$. Since $L_+(\Lambda Q)= -2Q$, we may take 
\begin{equation} \label{eq formula of T_j^1}
    T_j^{(1)}= - \sum_{k \neq j} \frac{\lambda_j^2 \Vert Q \Vert _{L^2}^2} {8\pi \lambda_k |\alpha_{jk}|} \Lambda Q
\end{equation}
to cancel the first two terms. This proves the conclusion when $N=1$.

Next, we construct $m_j^{(N+1)}, b_j^{(N+1)}$ and $T_j^{(N+1)}$ from the first $N$ terms. We have
\begin{equation} \begin{aligned}
    \tilde{E}_j^{(N+1)}-\tilde{E}_j^{(N)} &= \Delta T_j^{(N+1)}- T_j^{(N+1)}- i\lambda_j m_j^{(N+1)} \Lambda Q- \lambda_j^3 b_j^{(N+1)} \cdot y_j Q \\
    &\quad - \phi_{Q^2} T_j^{(N+1)}- 2\phi_{\mathrm{Re} \left( Q \overline{T_j^{(N+1)}} \right)} Q- \sum_{k \neq j} \psi_{Q^2,k}^{(N+1)} Q+ error \\
    &= -\Big( L_+ X_j^{(N+1)}+ \lambda_j^3 b_j^{(N+1)} \cdot y_j Q+ \sum_{k \neq j} \psi_{Q^2,k}^{(N+1)} Q \Big) \\
    &\quad -i\left( L_- Y_j^{(N+1)}+ \lambda_j m_j^{(N+1)} \Lambda Q \right)+ error,
\end{aligned} \end{equation}
where
\begin{equation}
    X_j^{(N+1)}= \mathrm{Re} \ T_j^{(N+1)}, \quad Y_j^{(N+1)}= \mathrm{Im} \ T_j^{(N+1)}, \quad error= I_1+ I_2+ I_3+ I_4,
\end{equation}
and
\begin{align} 
    I_1 &= -i \lambda_j m_j^{(N+1)} \Lambda \left( V_j^{(N+1)}-Q \right)- i\lambda_j M_j^{(N)} \Lambda T_j^{(N+1)} \\
    &\quad - \lambda_j^3 b_j^{(N+1)} \cdot y_j \left( V_j^{(N+1)}-Q \right)- \lambda_j^3 B_j^{(N)} \cdot y_j T_j^{(N+1)},
\end{align}
\begin{align}
    I_2 &= -\phi_{\left| V_j^{(N)} \right|^2- Q^2} T_j^{(N+1)}- 2\phi_{\mathrm{Re} \left( V_j^{(N)} \overline{T_j^{(N+1)}} \right)} \left( V_j^{(N+1)}- Q \right) \\
    &\quad - 2\phi_{\mathrm{Re} \left( \left( V_j^{(N)}- Q \right) \overline{T_j^{(N+1)}} \right)} Q- \phi_{\left| T_j^{(N+1)} \right|^2} V_j^{(N+1)},
\end{align}
\begin{align}
    I_3 &= -\sum_{k \neq j} \Bigg( \psi_{\left| V_k^{(N)} \right|^2- Q^2}^{(N+1)} V_j^{(N+1)}+ \psi_{Q^2,k}^{(N+1)} \left( V_j^{(N+1)}-Q \right)+ \phi_{\left| V_k^{(N)} \right|^2}^{(N)} T_j^{(N+1)} \\
    &\qquad \qquad \qquad \qquad \qquad \qquad + 2\phi_{\mathrm{Re} \left(V_k^{(N)} \overline{T_k^{(N+1)}} \right)}^{(N)} V_j^{(N+1)}+  \phi_{\left| T_k^{(N+1)} \right|^2}^{(N)} V_j^{(N+1)} \Bigg),
\end{align}
\begin{align}
    I_4 &= i\lambda_j^2 \sum_{k=1}^m \Bigg( \frac{\partial T_j^{(N+1)}}{\partial \alpha_k} \cdot 2\beta_k+ \frac{\partial T_j^{(N+1)}}{\partial \beta_k} \cdot B_k^{(N)}+ \frac{\partial V_j^{(N+1)}}{\partial \beta_k} \cdot b_k^{(N+1)} \\
    &\qquad \qquad \qquad \qquad \qquad \qquad \quad +\frac{\partial T_j^{(N+1)}}{\partial \lambda_k} \cdot M_k^{(N)}+ \frac{\partial V_j^{(N+1)}}{\partial \lambda_k} \cdot m_k^{(N+1)} \Bigg).
\end{align}

Assume $m_j^{(N+1)}, b_j^{(N+1)} \in S_{N+1}$ and $T_j^{(N+1)}$ is admissible of degree $N+1$. Using Lemma \ref{lem properties of admissible functions}, we see $error$ is admissible of degree $\ge N+2$. Thus it suffices to require
\begin{equation} 
    \left\{ \begin{aligned}
        &L_+ X_j^{(N+1)}= -\lambda_j^3 b_j^{(N+1)} \cdot y_j Q- \sum_{k \neq j} \psi_{Q^2,k}^{(N+1)} Q+ \mathrm{Re} \ \hat{E}_j^{(N)}, \\
        &L_- Y_j^{(N+1)}= -\lambda_j m_j^{(N+1)} \Lambda Q+ \mathrm{Im} \ \hat{E}_j^{(N)},
    \end{aligned} \right.
\end{equation}
where $\hat{E}_j^{(N)}$ is the sum of terms in $\tilde{E}_j^{(N)}$ that are admissible of degree $N+1$. 

Recall that we let the right hand sides be functions of $y_j$, so they are admissible of degree $N+1$. By Lemma \ref{lem inverse of L+, L-}, it suffices to require
\begin{equation} 
    \left\{ \begin{aligned}
        &\Big( \lambda_j^3 b_j^{(N+1)} \cdot y_j Q+ \sum_{k \neq j} \psi_{Q^2,k}^{(N+1)} Q- \mathrm{Re} \ \hat{E}_j^{(N)}, \nabla Q \Big)=0, \\
        &\left( \lambda_j m_j^{(N+1)} \Lambda Q- \mathrm{Im} \ \hat{E}_j^{(N)}, Q \right)=0.
    \end{aligned} \right.
\end{equation}
Such $b_j^{(N+1)}$ and $m_j^{(N+1)}$ exist because $(y_jQ, \nabla Q) \neq 0$ and $(\Lambda Q, Q) \neq 0$.
\end{proof}

\subsection{Accuracy of approximate solutions}

We verify the accuracy of $R^{(N)}$ as an approximate solution where $V_j^{(N)}$ is determined in Proposition \ref{prop construction of approximate bubbles}. We start with some estimates following from the definition of admissible functions. 

Let $\tilde{\Omega}= \Omega \times (\mathbb{R} / 2\pi \mathbb{Z})^m$ denote the space of modulation parameters and $g \in \tilde{\Omega}$. If $K$ is a compact set in $\mathbb{R}^{3m} \times \mathbb{R}_+^m$ and $(\beta, \gamma) \in K$, then by \eqref{eq decay of Q} and \eqref{eq estimate of admissible functions}, we have
\begin{equation} \label{eq decay of V_j}
    |V_j^{(N)}| \le Ce^{-c|y_j|}+ C_N a^{-1} e^{-c_N|y_j|},
\end{equation}
which also yields
\begin{equation} \label{eq decay of R_j}
    |R_{j,g}^{(N)}| \le Ce^{-c|x-\alpha_j|}+ C_N a^{-1} e^{-c_N|x-\alpha_j|}.
\end{equation}
By definition, $R_g^{(N)}$ is $C^1$ in $g$. Thus, if $g' \in \Omega$ and $(\beta', \gamma') \in K$, then
\begin{equation} \label{eq dependence of R in parameters}
    \left\Vert R^{(N)}_g- R^{(N)}_{g'} \right\Vert_{H^1} \le C_N \Vert g-g' \Vert. 
\end{equation} 
Finally, since $m_j^{(1)}= b_j^{(1)}=0$, we have
\begin{equation} \label{eq estimate of M_j, B_j}
    |M_j^{(N)}| \le Ca^{-2}+ C_N a^{-3}, \quad |B_j^{(N)}| \le Ca^{-2}+ C_N a^{-3}.
\end{equation}
Here, the capital constants depend on $K$, while the little ones do not.

The next lemma consists of two localization properties. The first item shows that the cross term about $R_j$ in \eqref{eq equation of approximate solution} does not matter. The second item will be used later.

\begin{lem} \label{lem localization}
Let $p \neq q \in \mathbb{R}^3$ and $u,v$ be functions such that
\begin{equation}
    |u(x)| \le e^{-|x-p|}, \quad |v(x)| \le e^{-|x-q|}, \quad \forall x \in \mathbb{R}^3.
\end{equation}
Then there exist absolute constants $C,c>0$ such that:
\begin{enumerate} [label=(\arabic*)]
    \item $\Vert \phi_{uv} \Vert_{L^\infty} \le C e^{-c|p- q|}$;
    \item If $f \in L^2$, then $\Vert \phi_{fu} fv \Vert_{L^1} \le C \max \Big\{ \frac{e^{-\frac{1}{2} |p-q|}} {|p-q|^{\frac{1}{2}}}, \frac{1}{|p-q|} \Big\} \Vert f \Vert_{L^2}^2$.
\end{enumerate}
\end{lem}

\begin{proof} \ \par

(1) Using the Hardy-Littlewood-Sobolev inequality, we get $\Vert \phi_{uv} \Vert_{L^\infty} \le C \Vert uv \Vert_{L^{3/2}}$. Note that either $|x-p| \ge \frac{1}{2} |p-q|$ or $|x-q| \ge \frac{1}{2} |p-q|$. In the first case, we use $|u(x)| \le e^{-c|p-q|}$ and $\Vert v \Vert_{L^{3/2}} \le C$ to conclude, and the second case is similar.

(2) We have
\begin{equation}
    \Vert \phi_{fu} fv \Vert_{L^1} \le C\iint \frac{|f(x)||f(y)|}{|x-y|} e^{-|x-p|-|y-q|} \d x \d y
\end{equation}
The integral on the region $|x-y| \ge \frac{1}{2} |p-q|$ is easily bounded by $\frac{C}{|p- q|} \Vert f \Vert_{L^2}^2$.

If $|x-y|< \frac{1}{2} |p-q|$, then $|x-p|+|y-q| \ge \frac{1}{2} |p-q|+ \frac{1}{4} |x-p|$, so by Cauchy-Schwarz,
\begin{equation} \begin{aligned}
    & \iint_{|x-y|< \frac{1}{2} |p-q|} \frac{|f(x)||f(y)|}{|x-y|} e^{-|x-p|-|y-q|} \d x \d y \\
    \le &\ e^{-\frac{1}{2} |p-q|} \int |f(x)| e^{-\frac{1}{4} |x-p|} \Big( \int_{|y-x| \le \frac{1}{2} |p-q|} \frac{|f(y)|}{|y-x|} \d y \Big) \d x\\
    \le &\ Ce^{-\frac{1}{2} |p-q|} |p-q|^{-\frac{1}{2}} \Vert f \Vert_{L^2} \int |f(x)| e^{-\frac{1}{4} |x-p|} \d x \le C \frac{e^{-\frac{1}{2} |p-q|}}{|p-q|^{\frac{1}{2}}} \Vert f \Vert_{L^2}^2. 
\end{aligned} \end{equation}
We then obtain the conclusion.
\end{proof}

The following is the main result in this subsection. It estimates the extent to which $R_g^{(N)}$, defined by \eqref{eq approximate solution}, satisfies the Hartree equation \eqref{eq hartree}.

\begin{prop} \label{prop accuracy of approximate solution}
Let $c_0, C_0>0$ and suppose $g \in C^1(\mathbb{R}_+, \tilde{\Omega})$ satisfies
\begin{equation} \label{eq boundedness of g}
    a \ge c_0, \quad |\beta| \le C_0, \quad c_0 \le \lambda_j \le C_0.
\end{equation}
Let $V_j^{(N)}$, $M_j^{(N)}$ and $B_j^{(N)}$ be as in Proposition \ref{prop construction of approximate bubbles}, $R_g^{(N)}$ be defined by \eqref{eq approximate solution}, and 
\begin{equation} \label{eq definition of Psi}
    \Psi^{(N)}= i\partial_t R_g^{(N)}+ \Delta R_g^{(N)}- \phi_{|R_g^{(N)}|^2} R_g^{(N)}- \sum_{j=1}^m \frac{1}{\lambda_j^4} S_j^{(N)} e^{-i\gamma_j+ i\beta_j \cdot x}.     
\end{equation}
Then there exist $c, C>0$ depending on $c_0$, $C_0$ and $N$ such that
\begin{equation} \label{eq estimate of Psi}
    |\Psi^{(N)}(t,x)| \le \frac{C}{a^{N+1}(t)}  \max \limits_j e^{-c|x-\alpha_j(t)|}.
\end{equation} 
\end{prop}

\begin{proof}
For simplicity, we omit the superscript $N$ and the subscript $g$.

By \eqref{eq equation of approximate solution}, we have
\begin{equation} \begin{aligned}
    \Psi &= \sum_{j=1}^m \frac{1}{\lambda_j^4} \tilde{E}_j(t,y_j) e^{-i\gamma_j+ i\beta_j \cdot x}+ \sum_{j \neq k} \phi_{\mathrm{Re} (R_j \overline{R_k})} R \\
    &\ +\sum_{j=1}^m \frac{1}{\lambda_j^4} \sum_{k \neq j} \left[ \phi_{\left| V_k \right|^2}^{(N)}- \left( \frac{\lambda_j}{\lambda_k} \right)^2 \phi_{|V_k|^2} \left( P(t), \frac{\lambda_j}{\lambda_k} y_j+ \frac{\alpha_{jk}}{\lambda_k} \right) \right] V_j e^{-i\gamma_j+ i\beta_j \cdot x}.
\end{aligned} \end{equation}
The first term is controlled using Proposition \ref{prop construction of approximate bubbles} and \eqref{eq estimate of admissible functions}. The second term is controlled using Lemma \ref{lem localization} and \eqref{eq decay of R_j}. For the last term, we claim that
\begin{equation} \label{eq accuaracy of dipole expansion, preparation}
    \left| \phi_{\left| V_k \right|^2}^{(N)}- \left( \frac{\lambda_j}{\lambda_k} \right)^2 \phi_{|V_k|^2} \left( \frac{\lambda_j}{\lambda_k} y_j+ \frac{\alpha_{jk}}{\lambda_k} \right) \right| \le \left\{ \begin{aligned}
        &C (1+|y_j|)^N,\ \lambda_j |y_j| \ge \frac{|\alpha_{jk}|}{3}, \\
        &\frac{C(1+|y_j|)^N}{|\alpha_{jk}|^{N+1}},\ \lambda_j |y_j| \le \frac{|\alpha_{jk}|}{3}.
    \end{aligned} \right.
\end{equation}
Since $F_n$ is homogeneous of degree $n$ in $\zeta$, using \eqref{eq decay of V_j}  and the Hardy-Littlewood-Sobolev inequality, we see the left hand side of \eqref{eq accuaracy of dipole expansion, preparation} is always bounded by $C (1+|y_j|)^N$, so we focus on the case when $\lambda_j |y_j| \le \frac{|\alpha_{jk}|}{3}$. We write
\begin{equation} \begin{aligned}
    LHS &= \frac{\lambda_j^2}{4\pi \lambda_k} \int |V_k(\xi)|^2 \Big| \frac{1}{|\alpha_{jk}- \lambda_k \xi+ \lambda_j y_j|}- \sum_{n=1}^N F_n(\alpha_{jk}, \lambda_k \xi- \lambda_j y_j) \Big| \d \xi \\
    &= \frac{\lambda_j^2}{4\pi \lambda_k} \left( \int_{\lambda_k |\xi| \ge \frac{|\alpha_{jk}|}{3}}+ \int_{\lambda_k |\xi| \le \frac{|\alpha_{jk}|}{3}} \right) \triangleq \frac{\lambda_j^2}{4\pi \lambda_k} (I_1+ I_2).
\end{aligned} \end{equation}
By Hardy-Littlewood-Sobolev, the definition of $F_n$ and \eqref{eq decay of V_j}, we have
\begin{equation} \begin{aligned}
    I_1 &\le C \left\Vert V_k(\xi) \chi_{\{ \lambda_k |\xi| \ge \frac{|\alpha_{jk}|}{3} \}} \right\Vert_{L^{8/3}}^2+ C \left\Vert |V_k(\xi)|^2 \big( 1+|\xi| \big)^N \chi_{\{ \lambda_k |\xi| \ge \frac{|\alpha_{jk}|}{3} \}} \right\Vert_{L^1} \\
    &\le C e^{-c|\alpha_{jk}|} \le \frac{C(1+|y_j|)^N}{|\alpha_{jk}|^{N+1}}.
\end{aligned} \end{equation}
By the Taylor formula, 
\begin{equation}
    \Big| \frac{1}{|\alpha- \zeta|}- \sum_{n=1}^N F_n(\alpha, \zeta) \Big| \le C \frac{|\zeta|^N}{|\alpha|^{N+1}}, \quad \text{if } |\zeta| \le \frac{|\alpha|}{3},
\end{equation}
so by the assumption that $\lambda_j |y_j| \le \frac{|\alpha_{jk}|}{3}$ and \eqref{eq decay of V_j}, we have
\begin{equation}
    I_2 \le \frac{C(1+|y_j|)^N}{|\alpha_{jk}|^{N+1}} \int |V_k(\xi)|^2 (1+|\xi|)^N \d \xi \le \frac{C(1+|y_j|)^N}{|\alpha_{jk}|^{N+1}}.
\end{equation}

Thus \eqref{eq accuaracy of dipole expansion, preparation} holds. Note that this yields
\begin{equation} \label{eq accuaracy of dipole expansion}
    \left| \phi_{\left| V_k \right|^2}^{(N)}- \left( \frac{\lambda_j}{\lambda_k} \right)^2 \phi_{|V_k|^2} \left( \frac{\lambda_j}{\lambda_k} y_j+ \frac{\alpha_{jk}}{\lambda_k} \right) \right| \le \frac{C(1+|y_j|)^{2N}}{a^{N+1}}.
\end{equation}
Then by \eqref{eq decay of V_j}, we can control the first term.
\end{proof}

From the estimate \eqref{eq estimate of Psi}, if $N$ is large, then the error $\Psi^{(N)}$ will decay rapidly in time. This helps us overcome the long range effect of \eqref{eq hartree}. This is also why we need to construct the approximate solution. From now on, the strategy becomes also similar to that of \cite{KsolitarywavesofNLS}.

\section{Reduction of the problem} \label{sec reduction}

Now we will focus on the hyperbolic case. Note that we do not have any additional assumption on the masses. We may still state the result in a more general way, for instance writing $a(t)$ instead of $t$, so that it is easier to apply it to the other two cases.

In this section, we perform two steps of reduction of the hyperbolic problem.

\subsection{Uniform estimates}

Due to \eqref{eq definition of Psi}, we want $S_j^{(N)}$ to vanish. Thus we need the following ODE result. It essentially states that the equation $S_j^{(N)}=0$ can be understood as a perturbation of the $m$-body problem \eqref{eq m-body problem}.

\begin{prop} \label{prop hyperbolic trajectory}
Let $P^{\infty}$ be a hyperbolic solution to \eqref{eq m-body problem} of the form \eqref{eq hyperbolic solution}, and $B_j^{(N)}, M_j^{(N)}$ be as determined in Proposition \ref{prop construction of approximate bubbles}. Then there exist $T_0= T_0(N)>0$ and $P^{(N)} \in C^1 \big( [T_0,+\infty), \Omega \big)$ such that
\begin{equation} \label{eq trajectory}
    \left\{ \begin{aligned}
        \dot{\alpha}_j^{(N)}(t) &= 2\beta_j^{(N)}(t), \\
        \dot{\beta}_j^{(N)}(t) &= B_j^{(N)} \left( P^{(N)}(t) \right), \\
        \dot{\lambda}_j^{(N)}(t) &= M_j^{(N)} \left( P^{(N)}(t) \right),
    \end{aligned} \right. \qquad \forall t \ge T_0
\end{equation}
and
\begin{equation} \label{eq estimate of hyperbolic trajectory}
    \left\Vert P^{(N)}(t)- P^\infty(t) \right\Vert \le t^{-1/2}, \qquad \forall t \ge T_0.
\end{equation}
\end{prop}

We need the exact expression of $b_j^{(2)}$. Since $T_j^{(1)}$ is real-valued, we have
\begin{equation} 
    \mathrm{Re} \ \hat{E}_j^{(1)} = -2\phi_{Q T_j^{(1)}} T_j^{(1)}- \phi_{ \left| T_j^{(1)} \right|^2} Q- \sum_{k \neq j} \left( 2\psi_{Q T_k^{(1)}}^{(1)} Q+ \psi_{Q^2,k}^{(1)} T_k^{(1)} \right).
\end{equation}
Since $T_j^{(1)}$ is even, by the explicit formula of $\psi^{(1)}$, $\mathrm{Re} \ \hat{E}_j^{(1)}$ is also even, and thus orthogonal to $\nabla Q$. We then obtain by the proof of Proposition \ref{prop construction of approximate bubbles} that 
\begin{equation}
    \Big( \lambda_j^3 b_j^{(2)} \cdot y_j Q+ \sum_{k \neq j} \psi_{Q^2,k}^{(2)} Q, \nabla Q \Big)=0.
\end{equation}
By the explicit formula of $\psi^{(2)}$ and using $Q$ is even, we deduce 
\begin{equation} \label{eq formula of b_j^2}
    b_j^{(2)}(P)= - \sum_{k \neq j} \frac{\Vert Q \Vert_{L^2}^2 \alpha_{jk}}{4\pi \lambda_k |\alpha_{jk}|^3}.
\end{equation}

Note that this is exactly the gravitational force acting on the $j$-th body. This explains the reason why we expect the $m$-body interaction and our choice of the coefficients. Also, \eqref{eq trajectory} can be viewed as a perturbation of the $m$-body equation \eqref{eq m-body problem}.

\begin{proof}[Proof of Proposition \ref{prop hyperbolic trajectory}]\ 

Let $\epsilon= \frac{1}{10}$ and $T_0>0$. Define the norm $\Vert \cdot \Vert_1$ of $P \in C \big( [T_0,+\infty), \Omega \big)$ by
\begin{equation}
    \Vert P \Vert_1:= \sum_{j=1}^m \sup_{t \ge T_0} \Big( t^{1-3\epsilon} |\alpha_j(t)|+ t^{2-2\epsilon} |\beta_j(t)|+t^{1-\epsilon} |\lambda_j(t)| \Big).
\end{equation}
Let $X= \Big\{ P \in C \big( [T_0,+\infty), \Omega \big) \ \big| \ \Vert P-P^\infty \Vert_1 \le 1 \Big\}$ and define $\Gamma P(t)$ by
\begin{equation} \begin{gathered}
    \Gamma \alpha_j(t)= \alpha_j^\infty(t)+ \int_t^\infty 2(\beta_j^\infty(\tau)- \beta_j(\tau)) \d \tau, \\
    \Gamma \beta_j(t)= \lim_{t \to \infty} \beta_j^\infty(t)- \int_t^\infty B_j^{(N)}(P(\tau)) \d \tau, \\
    \Gamma \lambda_j(t)= \lambda_j^\infty- \int_t^\infty M_j^{(N)}(P(\tau)) \d \tau.
\end{gathered} \end{equation}
Because of the decay of $a$ in $t$, we know $\lim \limits_{t \to \infty} \beta_j^\infty(t)$ does exist.

We claim that: if $T_0(N)$ is large enough, then $\Gamma$ maps $X$ into $X$, and for $P,P' \in X$, we have $\Vert \Gamma P- \Gamma P' \Vert_1 \le \frac{1}{2} \Vert P- P' \Vert_1$.

Assume $P \in X$. Since $P^\infty$ is hyperbolic, we have $a(t) \gtrsim t$. First,
\begin{equation}
    |\Gamma \alpha_j(t)- \alpha_j^\infty(t)| \le 2 \int_t^\infty |\beta_j(\tau)- \beta_j^\infty(\tau)| \d \tau \le 2 \int_t^\infty \tau^{2\epsilon-2} \d \tau \le C t^{2\epsilon-1}.
\end{equation}
Using $b_j^{(1)}=0$ and $b_j^{(2)}(P^\infty(t)) = \dot{\beta}_j^\infty(t)$, which comes from \eqref{eq formula of b_j^2}, we have
\begin{equation}
    |\Gamma \beta_j(t)- \beta_j^\infty(t)| \le \int_t^\infty \big| b_j^{(2)}(P(\tau))-b_j^{(2)}(P^\infty(\tau)) \big| \d \tau+ \sum_{n=3}^N \int_t^\infty \big| b_j^{(n)}(P(\tau)) \big| \d \tau.
\end{equation}
By the fundamental theorem of calculus, we have
\begin{equation}
    \big| b_j^{(2)}(P(\tau))-b_j^{(2)}(P^\infty(\tau)) \big| \le C \left( \frac{|\alpha- \alpha^\infty|}{a^3}+ \frac{|\beta- \beta^\infty|}{a^2}+ \frac{|\lambda- \lambda^\infty|}{a^2} \right) \le C \tau^{\epsilon-3},
\end{equation}
and thus, using $b_j^{(n)} \in S_n$, we get
\begin{equation}
    |\Gamma \beta_j(t)- \beta_j^\infty(t)| \le C \int_t^\infty \tau^{\epsilon-3} \d \tau+ C_N\sum_{n=3}^N \int_t^\infty \tau^{-n} \d \tau \le C_N t^{\epsilon-2}.
\end{equation}
Using $m_j^{(1)}=0$ and $m_j^{(n)} \in S_n$, we have
\begin{equation} \begin{aligned}
    |\Gamma \lambda_j(t)- \lambda_j^\infty(t)| \le \sum_{n=2}^N \int_t^\infty \big| m_j^{(n)}(P(\tau)) \big| \d \tau \le C_N \sum_{n=2}^N \int_t^\infty \tau^{-n} \d \tau \le C_N t^{-1}.
\end{aligned} \end{equation}
Collecting the above estimates, we get $\Vert \Gamma P \Vert_1 \le C_N T_0^{-\epsilon}$.

Thus for $T_0(N)$ large enough, we have $\Gamma: X \to X$. The contraction property can be checked in the same way. By the contraction mapping theorem, $\Gamma$ has a unique fixed point in $X$. Taking this fixed point as $P^{(N)}$, then the requirements are satisfied.
\end{proof}

From Proposition \ref{prop accuracy of approximate solution} and \ref{prop hyperbolic trajectory}, we know $R_{g^{(N)}} ^{(N)}$ is almost a solution of \eqref{eq hartree}. We then reduce the hyperbolic case to the following uniform estimate with a bootstrap assumption.

\begin{prop} \label{prop uniform estimate}
Let $P^{(N)}$ be defined as in Proposition \ref{prop hyperbolic trajectory} and $\gamma_j^{(N)}(t)$ be such that
\begin{equation}
    \gamma_j^{(N)}(0)=0, \quad \dot{\gamma}_j^{(N)}(t)= -\frac{1}{\lambda_j^{(N)}(t)^2} + |\beta_j^{(N)}(t)|^2+ \dot{\beta}_j^{(N)}(t) \cdot \alpha_j^{(N)}(t).
\end{equation}

Let $T_n \to +\infty$ and $u_n$ be the solution to 
\begin{equation} \label{eq equation of un}
    \left\{ \begin{aligned}
        &i\partial_t u_n+ \Delta u_n- \phi_{|u_n|^2} u_n= 0, \\
        &u_n(T_n,\cdot)= R_{g^{(N)}} ^{(N)}(T_n,\cdot).
    \end{aligned} \right.
\end{equation}
Then there exists $T_0=T_0(N)$ such that for $N$ large and $T_* \in [T_0,T_n]$, if
\begin{equation} \label{eq uniform estimate, bootstrap}
    \left\Vert u_n(t)- R_{g^{(N)}}^{(N)}(t) \right\Vert_{H^1} \le 2t^{-\frac{N}{9}}, \qquad \forall n\ge 1,\ \forall t \in [T_*,T_n],
\end{equation}
then
\begin{equation}
    \left\Vert u_n(t)- R_{g^{(N)}}^{(N)}(t) \right\Vert_{H^1} \le t^{-\frac{N}{9}}, \qquad \forall n\ge 1,\ \forall t \in [T_*,T_n].
\end{equation}
\end{prop}

\begin{proof} [Proof of the hyperbolic case by Proposition \ref{prop uniform estimate}] \ \par
Fix a large $N$ such that the conclusion holds. By the standard bootstrap argument, we know \eqref{eq uniform estimate, bootstrap} actually holds with $T_*=T_0$. Using \eqref{eq decay of R_j}, we know there exists $C>0$ such that
\begin{equation} \label{eq uniform estimate of un, boundedness}
    \Vert u_n(t) \Vert_{H^1} \le C, \quad \forall n \ge 1,\ \forall t \in [T_0,T_n].
\end{equation}
Also, for any $\delta>0$, there exist $r=r(\delta)>0$ and $t_0=t_0(\delta)>T_0$ such that 
\begin{equation}
    \int_{|x|>r} |u_n(t_0,x)|^2 \d x< \delta, \quad \forall n \ge 1.
\end{equation}
We claim that there exists $r'=r'(\delta)>0$ such that
\begin{equation} \label{eq uniform estimate of un, decay}
    \int_{|x|>r'} |u_n(T_0,x)|^2 \d x< 2\delta, \quad \forall n \ge 1.
\end{equation}

To prove \eqref{eq uniform estimate of un, decay}, let $\Phi \in C^\infty(\mathbb{R})$ be a cutoff such that 
\begin{equation}
    0 \le \Phi \le 1, \quad 0 \le \Phi' \le 2, \quad \Phi(x)= \left\{ \begin{aligned}
        0, &&x \le 0, \\
        1, &&x \ge 1.
    \end{aligned} \right. 
\end{equation}
Let $L>0$ and define $z(t)= \int |u_n(t,x)|^2 \Phi( \frac{|x|-r}{L}) \d x$. Then $z(t_0) \le \delta$. Since
\begin{equation}
    z'(t)= -2 \mathrm{Im} \int \Delta u_n \overline{u_n} \Phi \Big( \frac{|x|-r}{L} \Big) \d x= \frac{2}{L} \mathrm{Im} \int \nabla u_n \overline{u_n} \cdot \frac{x}{|x|} \Phi' \Big( \frac{|x|-r}{L} \Big) \d x,
\end{equation}
we have $|z'(t)| \le \frac{4}{L} \Vert u_n \Vert_{H^1}^2$. Integrating in $t$ and using \eqref{eq uniform estimate of un, boundedness}, we get 
\begin{equation}
    \int_{|x|>L+r} |u_n(T_0,x)|^2 \d x \le z(T_0) \le \frac{4C^2(t_0-T_0)}{L}+ \delta.
\end{equation}
We deduce \eqref{eq uniform estimate of un, decay} by taking $L=L(\delta)$ large enough and $r'=L+r$.

Now, \eqref{eq uniform estimate of un, boundedness} and \eqref{eq uniform estimate of un, decay} imply the existence of a subsequence $u_{n_k}(T_0)$ of $u_n(T_0)$ that converges in $L^2$ to some $U_0$ and $U_0 \in H^1$. Let $U$ be the solution to
\begin{equation} 
    \left\{ \begin{aligned}
        &i\partial_t U+ \Delta U- \phi_{|U|^2} U= 0, \\
        &U(T_0)=U_0.
    \end{aligned} \right.
\end{equation}
By global well-posedness of the equation, we have $u_{n_k}(t) \to U(t)$ in $L^2$ for any $t \ge T_0$. Thanks to \eqref{eq uniform estimate of un, boundedness}, by passing to subsequence, we may assume $u_{n_k}(t) \rightharpoonup U(t)$ in $H^1$. Using \eqref{eq uniform estimate, bootstrap} and Fatou's lemma, we deduce
\begin{equation}
    \left\Vert U(t)- R_{g^{(N)}}^{(N)}(t) \right\Vert_{H^1} \le 2t^{-\frac{N}{9}}, \quad \forall t \ge T_0.
\end{equation}

Let $\gamma^\infty(t)$ be such that
\begin{equation}    
    \gamma_j^\infty(0)=0, \quad \dot{\gamma}_j^\infty(t)= -\frac{1}{(\lambda_j^\infty)^2} + |\beta_j^\infty(t)|^2+ \dot{\beta}_j^\infty(t) \cdot \alpha_j^\infty(t).
\end{equation}
Then by \eqref{eq dependence of R in parameters}, \eqref{eq estimate of hyperbolic trajectory} and \eqref{eq estimate of admissible functions}, we obtain the conclusion of the hyperbolic case.
\end{proof}

\subsection{Modulation estimates}

We want to find a family of modulation parameters $\alpha, \beta, \lambda$ and $\gamma$ such that $R_g^{(N)}$ is an orthogonal projection of $u_n$. More precisely, we prove the following lemma.

\begin{lem} \label{lem orthogonality}
Let $N,n \ge 1$. Then there exist $T_0=T_0(N)>0$ and a unique modulation parameter $g \in C^1([T_0,+\infty), \tilde{\Omega})$ such that: if 
\begin{equation}
    \varepsilon(t,x)= u_n(t,x)- R_g^{(N)}(t,x),
\end{equation}
then for $t \ge T_0$ and $1\le j \le m$, we have
\begin{equation} \label{eq orthogonality} \begin{aligned}
    &\mathrm{Re} \Big( \varepsilon(t), g_j V_j^{(N)} \Big)= \mathrm{Re} \Big( \varepsilon(t), g_j \big( y_j V_j^{(N)} \big) \Big) \\
    = &\ \mathrm{Im} \Big( \varepsilon(t), g_j \big( \Lambda V_j^{(N)} \big) \Big)= \mathrm{Im} \Big( \varepsilon(t), g_j \big( \nabla V_j^{(N)}\big) \Big)= 0.    
\end{aligned} \end{equation}
In particular, we have
\begin{equation} \label{eq value of epsilon at T_n}
    g(T_n)=g^{(N)}(T_n), \quad \varepsilon(T_n)=0.
\end{equation}
\end{lem}

To prove the above result, we first work on a time-independent version.

\begin{lem} \label{lem orthogonality, time independent version}
Let $N \ge 1$ and $K \subset \mathbb{R}_+^m$ be compact. Then there exist $\delta, A>0$ such that: if $g^0 \in \tilde{\Omega}$ and $u \in H^1$ satisfy $a^0>A$, $\lambda^0 \in K$ and $\big\Vert u- R_{g^0}^{(N)} \big\Vert_{H^1} < \delta$, then there exists a unique parameter $g \in \tilde{\Omega}$ that $C^1$-depends on $u$ and
\begin{equation} \begin{aligned}
    &\mathrm{Re} \Big( u-R_g^{(N)}, g_j V_j^{(N)} \Big)= \mathrm{Re} \Big( u-R_g^{(N)}, g_j \big( y_j V_j^{(N)} \big) \Big) \\
    = &\ \mathrm{Im} \Big( u-R_g^{(N)}, g_j \big( \Lambda V_j^{(N)} \big) \Big)= \mathrm{Im} \Big( u-R_g^{(N)}, g_j \big( \nabla V_j^{(N)}\big) \Big)= 0.    
\end{aligned} \end{equation} 
\end{lem}

\begin{proof}
Let $p=(g,u)$ and $\varepsilon(p)= u- R_g^{(N)}$. Set $u_0= R_{g^0}^{(N)}$ and $p_0= (g_0,u_0)$. Define
\begin{equation} \begin{gathered}
    \rho_j^1(p)= \mathrm{Re} \Big( \varepsilon(p), g_j V_j^{(N)} \Big), \quad \rho_j^2(p)= \mathrm{Re} \Big( \varepsilon(p), g_j \big( y_j V_j^{(N)} \big) \Big), \\
    \rho_j^3(p)= \mathrm{Im} \Big( \varepsilon(p), g_j \big( \Lambda V_j^{(N)} \big) \Big), \quad \rho_j^4(p)= \mathrm{Im} \Big( \varepsilon(p), g_j \big( \nabla V_j^{(N)}\big) \Big). 
\end{gathered} \end{equation}
Then $\varepsilon(p_0)=0$ and $\rho_j^\nu(p_0)=0$ for $\nu=1,2,3,4$.

We would like to compute $\frac{\partial \rho_j^\nu}{\partial g} (p_0)$. Since $\varepsilon(p_0)=0$, the partial derivative of $\rho_j^\nu$ evaluated at $p_0$ only falls on $\varepsilon$. We compute
\begin{equation} \begin{aligned}
    \frac{\partial \varepsilon}{\partial \alpha_j} &= -\frac{1}{\lambda_j} g_j \big( \nabla V_j^{(N)} \big)+ g_j \Big( \frac{\partial V_j^{(N)}}{\partial \alpha_j} \Big), \\
    \frac{\partial \varepsilon}{\partial \beta_j} &= i \alpha_j g_j V_j^{(N)}+ i\lambda_j g_j \big( y_j V_j^{(N)} \big)+ g_j \Big( \frac{\partial V_j^{(N)}}{\partial \beta_j} \Big), \\
    \frac{\partial \varepsilon}{\partial \lambda_j} &= g_j \big( \Lambda V_j^{(N)} \big)+ g_j \Big( \frac{\partial V_j^{(N)}}{\partial \lambda_j} \Big), \\
    \frac{\partial \varepsilon}{\partial \gamma_j} &= -i g_j V_j^{(N)}.
\end{aligned} \end{equation}
Using \eqref{eq estimate of admissible functions}, \eqref{eq decay of V_j} and that $Q$ is real and even, we can represent $\frac{\partial \rho_j^\nu}{\partial g} (p_0)$ by
\begin{center} 
\begin{tabular}{c|c c c c}
         &1 &2 &3 &4 \\ \hline
    $\alpha$ &0 &1 &0 &0 \\
    $\beta$ &0 &0 &* &1 \\
    $\lambda$ &1 &0 &0 &0 \\
    $\gamma$ &0 &0 &1 &0 \\
\end{tabular} 
\end{center}
where $0$, for instance the $(\alpha,1)$ entry, represents 
\begin{equation}
    \frac{\partial \rho_j^1}{\partial \alpha_k}(p_0)= o(1), \quad \forall j,k,
\end{equation}
while $1$, for instance the $(\alpha,2)$ entry, represents 
\begin{equation}
    \frac{\partial \rho_j^2}{\partial \alpha_k}(p_0)= \left\{ \begin{matrix}
        o(1), & j\neq k, \\
        c_j+ o(1), & j=k.
    \end{matrix} \right.
\end{equation}
Here $c_j$ is invertible and independent of $a$, and $o(1)$ means goes to $0$ as $a \to +\infty$.

Therefore, for $A$ large enough, $\frac{\partial \rho_j^\nu}{\partial g} (p_0)$ is an invertible matrix. Then we can conclude by the implicit function theorem. The last comment is that $g \in \tilde{\Omega}$ because $g$ is closed to $g^{(N)}$, which means we have $a>\frac{A}{2}$ when $\delta$ is small.
\end{proof}

\begin{proof}[Proof of Lemma \ref{lem orthogonality}] \ \par
By \eqref{eq uniform estimate, bootstrap}, for $T_0(N)$ large enough and $\delta, A$ determined in Lemma \ref{lem orthogonality, time independent version}, if $t \ge T_0$, then $a^{(N)}(t)>A$ and $\big\Vert u_n(t)- R_{g^{(N)}}^{(N)}(t) \big\Vert_{H^1} < \delta$. By Lemma \ref{lem orthogonality, time independent version}, there exists a unique $g(t) \in \tilde{\Omega}$ such that \eqref{eq orthogonality} holds. Moreover, $g \in C^1$ because $g$ is $C^1$ in $u_n$ and $u_n$ is $C^1$ in $t$.
\end{proof}

It follows from the implicit function theorem that $g$ is closed to $g^{(N)}$. But to prove Proposition \ref{prop uniform estimate}, we need a quantitative estimate of $g-g^{(N)}$ and $\varepsilon$. 

\begin{prop} \label{prop bootstrap}
For $N$ and $T_0=T_0(N)$ large enough, $\forall n\ge 1,\ T_* \in [T_0,T_n]$, if
\begin{equation} \label{eq bootstrap assumption}
   \left\{ \begin{aligned}
       &\Vert \varepsilon(t) \Vert_{H^1} \le t^{-\frac{N}{4}}, \\
       &\sum_{j=1}^m \left| \lambda_j(t)- \lambda_j^{(N)}(t) \right|+ \left| \beta_j(t)- \beta_j^{(N)}(t) \right| \le t^{-1-\frac{N}{8}}, \\
       &\sum_{j=1}^m \left| \gamma_j(t)- \gamma_j^{(N)}(t) \right|+ \left| \alpha_j(t)- \alpha_j^{(N)}(t) \right| \le t^{-\frac{N}{8}},
   \end{aligned} \right. \qquad \forall t \in [T_*,T_n],
\end{equation}
then
\begin{equation} \label{eq bootstrap conclusion}
   \left\{ \begin{aligned}
       &\Vert \varepsilon(t) \Vert_{H^1} \le \frac{1}{2} t^{-\frac{N}{4}}, \\
       &\sum_{j=1}^m \left| \lambda_j(t)- \lambda_j^{(N)}(t) \right|+ \left| \beta_j(t)- \beta_j^{(N)}(t) \right| \le \frac{1}{2} t^{-1-\frac{N}{8}}, \\
       &\sum_{j=1}^m \left| \gamma_j(t)- \gamma_j^{(N)}(t) \right|+ \left| \alpha_j(t)- \alpha_j^{(N)}(t) \right| \le \frac{1}{2} t^{-\frac{N}{8}},
   \end{aligned} \right. \qquad \forall t \in [T_*,T_n].
\end{equation}
\end{prop}

\begin{proof} [Proof of Proposition \ref{prop uniform estimate} by Proposition \ref{prop bootstrap}] \ \par
As the left hand sides are continuous in $t$, by a bootstrap argument, we know \eqref{eq bootstrap assumption} actually holds for any $t \in [T_0,T_n]$. Then by \eqref{eq dependence of R in parameters}, we have
\begin{equation} \begin{aligned}
    \left\Vert u_n(t)- R_{g^{(N)}}^{(N)} \right\Vert_{H^1} &\le \left\Vert u_n(t)- R_g^{(N)} \right\Vert_{H^1}+ \left\Vert R_g^{(N)}- R_{g^{(N)}}^{(N)} \right\Vert_{H^1} \\
    &\le \Vert \varepsilon(t) \Vert_{H^1}+ C_N \Vert g- g^{(N)} \Vert \le t^{-\frac{N}{4}}+ C_N t^{-\frac{N}{8}} \le t^{-\frac{N}{9}}
\end{aligned} \end{equation}
when $T_0(N)$ is large enough. This finishes the proof of Proposition \ref{prop uniform estimate}.
\end{proof}

So far, we have reduced the hyperbolic case to Proposition \ref{prop bootstrap}.

\section{Estimates of the modulation} \label{sec estimate}

To simplify notations, we write $R_j= R_{j,g}^{(N)}$, $R= R_g^{(N)}$, $M_j= M_j^{(N)}$, $B_j= B_j^{(N)}$, $V_j= V_j^{(N)}$, $S_j= S_j^{(N)}$ and $\Psi= \Psi^{(N)}$. But we will make clear whether a constant depends on $N$.

By \eqref{eq equation of un} and \eqref{eq definition of Psi}, we have
\begin{equation} \label{eq equation of epsilon}
    i\partial_t \varepsilon+ \Delta \varepsilon- 2\phi_{\mathrm{Re} (\varepsilon \overline{R})}R- \phi_{|R|^2} \varepsilon= \N(\varepsilon)- \Psi- \sum_{j=1}^m \frac{1}{\lambda_j^4} S_j(t,x) e^{-i\gamma_j+ i\beta_j \cdot x},
\end{equation} 
where
\begin{equation}
    \N(\varepsilon)= 2\phi_{\mathrm{Re} (\varepsilon \overline{R})} \varepsilon+ \phi_{|\varepsilon|^2}R+ \phi_{|\varepsilon|^2} \varepsilon.
\end{equation}
By the Sobolev inequality and the Hardy-Littlewood-Sobolev inequality, we have
\begin{equation} \label{eq estimate of N(epsilon)}
    \Vert \N(\varepsilon) \Vert_{L^2} \le C_N \Vert \varepsilon \Vert_{H^1}^2.
\end{equation}

By \eqref{eq hyperbolic solution}, \eqref{eq estimate of hyperbolic trajectory}, \eqref{eq bootstrap assumption} and \eqref{eq estimate of M_j, B_j}, we have the following asymptotic properties:
\begin{equation} \label{eq asymptotic behavior of hyperbolic trajectory}
    a(t) \sim t, \quad |\alpha_j| \lesssim t, \quad |\beta_j| \lesssim 1, \quad |\dot{\beta}_j| \lesssim \frac{1}{t^2}, \quad \lambda_j \sim 1, \quad |\dot{\lambda}_j| \lesssim \frac{1}{t^2}.
\end{equation}
In particular, \eqref{eq boundedness of g} is satisfied. 

\subsection{Control of the parameters} \label{subsec parameter}

In this subsection, we aim at proving the second and the third line of \eqref{eq bootstrap conclusion}.

Define the modulation error 
\begin{equation} \begin{aligned}
    Mod(t)= \sum_{j=1}^m \Bigg( &\Big| \dot{\alpha}_j(t)- 2\beta_j(t) \Big| + \Big| \dot{\beta}_j(t)- B_j(P(t)) \Big| + \Big| \dot{\lambda}_j(t)- M_j(P(t)) \Big| \\
    &+ \bigg| \dot{\gamma}_j(t)+ \frac{1}{\lambda_j^2(t)}- |\beta_j(t)|^2- \dot{\beta}_j(t) \cdot \alpha_j(t) \bigg| \Bigg). 
\end{aligned} \end{equation}
First notice that $S_j$ is controlled by $Mod$:
\begin{equation} \label{eq estimate of S_j with Mod}
    |S_j(t,x)| \le C_N Mod(t) e^{-c_N|x-\alpha_j(t)|}.
\end{equation}

Let $\theta(t,x)$ be a function such that
\begin{equation} \label{eq decay of theta}
    |\theta(t,x)| \le C e^{-c|x|}+ C_N a^{-1} e^{-c_N|x|}, \quad \forall t>0,\ x \in \mathbb{R}^3,
\end{equation}
which corresponds to \eqref{eq decay of V_j}, \eqref{eq decay of R_j}, and let $\theta_j= g_j \theta$.

Using \eqref{eq equation of epsilon} and integration by parts, we can compute
\begin{equation} \begin{aligned}
    \frac{\d}{\dt} \mathrm{Im} \int \varepsilon \overline{\theta_j}= &\ \mathrm{Re} \int \varepsilon \left( \overline{i\partial_t \theta_j+ \Delta \theta_j- 2\phi_{\mathrm{Re} (\theta_j \overline{R})}R- \phi_{|R|^2} \theta_j} \right) \\
    &+ \mathrm{Re} \int (\Psi- \N(\varepsilon)) \overline{\theta_j}+ \sum_{j=1}^m \mathrm{Re} \int \frac{1}{\lambda_j^4} S_j(t,x) e^{-i\gamma_j+ i\beta_j \cdot x} \overline{\theta_j}.
\end{aligned} \end{equation}
By \eqref{eq basic calculation}, \eqref{eq estimate of M_j, B_j} and \eqref{eq decay of theta}, we have
\begin{equation} \begin{aligned}
    i\partial_t \theta_j+ \Delta \theta_j &= \frac{1}{\lambda_j^4} \left( i\lambda_j^2 \partial_t \theta+ \Delta \theta- \theta \right) e^{-i\gamma_j+ i\beta_j \cdot x} \\
    &\quad + O \left( \frac{1}{a^2}+ \frac{C_N}{a^3}+ Mod \right) \left( e^{-c |x-\alpha_j|}+ \frac{C_N}{a} e^{-c_N|x-\alpha_j|} \right).
\end{aligned} \end{equation}
By Lemma \ref{lem localization}, \eqref{eq decay of R_j} and \eqref{eq decay of theta}, we have
\begin{equation} \begin{aligned}
    \phi_{\mathrm{Re} (\theta_j \overline{R})} R &= \frac{1}{\lambda_j^4} \phi_{\mathrm{Re} (\theta \overline{V_j})} V_j e^{-i\gamma_j+ i\beta_j \cdot x}+ \phi_{\mathrm{Re} (\theta_j \overline{R}_j)} \sum_{k \neq j} R_k \\
    &\quad + O_N \Big( e^{-c_N a} \max_k e^{-c_N|x-\alpha_k|} \Big).
\end{aligned} \end{equation}
By Lemma \ref{lem localization}, \eqref{eq decay of R_j}, \eqref{eq decay of theta}, \eqref{eq accuaracy of dipole expansion} and the explicit formula of $\psi^{(2)}$, we have
\begin{equation} \begin{aligned}
    \phi_{|R|^2} \theta_j &= \frac{1}{\lambda_j^4} \bigg( \phi_{|V_j|^2}+ \sum_{k \neq j} \psi_{|V_k|^2}^{(1)} \bigg) \theta e^{-i\gamma_j+ i\beta_j \cdot x}+ O \left( \frac{1}{a^2} e^{-c |x-\alpha_j|} \right) \\
    &\quad + O_N \Big( (e^{-c_N a}+a^{-3}) e^{-c_N|x-\alpha_j|} \Big).
\end{aligned} \end{equation}

We collect the terms of degree $1$ in $\theta$
\begin{equation}
    L_j \theta:= -\Delta \theta+ \theta+ 2\phi_{\mathrm{Re} (\theta \overline{V_j})} V_j+ \bigg( \phi_{|V_j|^2}+ \sum_{k \neq j} \psi_{|V_k|^2}^{(1)} \bigg) \theta. 
\end{equation}
With \eqref{eq asymptotic behavior of hyperbolic trajectory}, if we take $T_0(N)$ large enough, then we have
\begin{equation} \begin{aligned}
    &i\partial_t \theta_j+ \Delta \theta_j- 2\phi_{\mathrm{Re} (\theta_j \overline{R})}R- \phi_{|R|^2} \theta_j \\
    =&\ \frac{1}{\lambda_j^4} \left( i\lambda_j^2 \partial_t \theta- L_j \theta \right) e^{-i\gamma_j+ i\beta_j \cdot x}- 2\phi_{\mathrm{Re} (\theta_j \overline{R}_j)} \sum_{k \neq j} R_k\\
    &\ +O \left( \frac{1}{a^2}+ Mod \right) e^{-c |x-\alpha_j|}+ O_N \left( \frac{1}{a^3}+ \frac{Mod}{a} \right) e^{-c_N \min_j |x-\alpha_j|}.
\end{aligned} \end{equation}
Inserting this into the previous formula, using \eqref{eq estimate of Psi}, \eqref{eq decay of theta}, \eqref{eq estimate of N(epsilon)}, and again taking $T_0(N)$ large enough, we get
\begin{equation} \begin{aligned}
    \frac{\d}{\dt} \mathrm{Im} \int \varepsilon \overline{\theta_j}= &\ \mathrm{Re} \int \frac{\varepsilon}{\lambda_j^4} \overline{\left( i\lambda_j^2 \partial_t \theta- L_j \theta \right) e^{-i\gamma_j+ i\beta_j \cdot x}}- 2\mathrm{Re} \int \varepsilon \phi_{\mathrm{Re} (\theta_j \overline{R}_j)} \sum_{k \neq j} \overline{R_k} \\
    &+ \frac{1}{\lambda_j^6} \mathrm{Re} \int S_j \overline{\theta}+ O \left( \frac{\Vert \varepsilon \Vert_{H^1}}{a^2}+ Mod \Vert \varepsilon \Vert_{H^1} \right)+ O_N \left( \frac{1}{a^{N+1}}+ \Vert \varepsilon \Vert_{H^1}^2 \right).  
\end{aligned} \end{equation}

We will take $\theta$ to be $iV_j$, $\nabla V_j$, $\Lambda V_j$ and $iy_j V_j$. By \eqref{eq orthogonality}, the left hand side always vanishes. By \eqref{eq decay of V_j}, \eqref{eq asymptotic behavior of hyperbolic trajectory} and \eqref{eq estimate of M_j, B_j}, we always have
\begin{equation} \label{eq estimate on theta: partial_t}
    \partial_t \theta= O \left( \frac{1}{a^2}+ \frac{C_N}{a^3}+ \frac{Mod}{a} \right) \left( e^{-c |x-\alpha_j|}+ \frac{C_N}{a} e^{-c_N|x-\alpha_j|} \right).
\end{equation}
By the proof of Proposition \ref{prop construction of approximate bubbles}, we know 
\begin{equation}
    W_j:= -\Delta V_j+ V_j+ \bigg( \phi_{|V_j|^2}+ \sum_{k \neq j} \psi_{|V_k|^2}^{(1)} \bigg) V_j
\end{equation}
is admissible of degree $\ge 2$. Direct computation yields
\begin{equation} \begin{gathered}
    L_j (iV_j) = iW_j, \quad L_j (\Lambda V_j) = (\Lambda+2) W_j- 2 \bigg(1+ \sum_{k \neq j} \psi_{|V_k|^2}^{(1)} \bigg) V_j, \\
    L_j (\nabla V_j) = \nabla W_j, \quad L_j (iy_j V_j) =iy_j W_j- 2i \nabla V_j.
\end{gathered} \end{equation}
By \eqref{eq estimate of admissible functions}, \eqref{eq orthogonality} and that $\psi^{(1)}$ is constant, we always have
\begin{equation}
    L_j \theta= f+ O\left( \frac{1}{a^2}+ \frac{C_N}{a^3} \right) \left( e^{-c |x-\alpha_j|}+ \frac{C_N}{a} e^{-c_N|x-\alpha_j|} \right),
\end{equation}
where $f$ is a function such that $\mathrm{Re} \int \varepsilon (\overline{g_j f})= 0$. Thus
\begin{equation} \label{eq estimate on theta: L_j}
    \mathrm{Re} \int \frac{\varepsilon}{\lambda_j^4} \overline{ L_j \theta\ e^{-i\gamma_j+ i\beta_j \cdot x}}= O \left( \frac{\Vert \varepsilon \Vert_{H^1}}{a^2}+ \frac{C_N \Vert \varepsilon \Vert_{H^1}}{a^3} \right).
\end{equation}
By \eqref{eq accuaracy of dipole expansion} and the explicit formula of $\psi^{(2)}$, we have
\begin{equation}
    \phi_{\mathrm{Re} (\theta_j \overline{R_j})}= \psi_{\mathrm{Re} (\theta_j \overline{R_j})}^{(1)}+ O \left( \frac{1}{a^2}+ \frac{C_N}{a^3} \right) \big( 1+|x-\alpha_j| \big)^2.
\end{equation}
Since $\psi^{(1)}$ is constant, by \eqref{eq orthogonality}, we always have
\begin{equation} \label{eq estimate on theta: phi}
    2\mathrm{Re} \int \varepsilon \phi_{\mathrm{Re} (\theta_j \overline{R}_j)} \sum_{k \neq j} \overline{R_k}= O \left( \frac{\Vert \varepsilon \Vert_{H^1}}{a^2}+ \frac{C_N \Vert \varepsilon \Vert_{H^1}}{a^3} \right).
\end{equation}
Finally, using \eqref{eq definition of S_j^N}, \eqref{eq decay of V_j} and that $Q$ is even, for $\theta$ taken as the four functions, 
\begin{equation} \label{eq estimate on theta: S_j}
    \sum_{\theta} \sum_{j=1}^m \left| \mathrm{Re} \int S_j \overline{\theta} \right| \ge c\ Mod- \frac{C_N Mod}{a}.
\end{equation}

Therefore, gathering \eqref{eq estimate on theta: partial_t}, \eqref{eq estimate on theta: L_j}, \eqref{eq estimate on theta: phi} and \eqref{eq estimate on theta: S_j}, we obtain
\begin{equation}
    Mod(t) \le \frac{C\Vert \varepsilon \Vert_{H^1}}{a^2}+ \frac{C_N \Vert \varepsilon \Vert_{H^1}}{a^3}+ \frac{C_N Mod(t)}{a}+ \frac{C_N}{a^{N+1}} + C_N\Vert \varepsilon \Vert_{H^1}^2. 
\end{equation}
Taking $T_0(N)$ large enough to absorb some $O_N$ terms, we get
\begin{equation} \label{eq estimate of Mod with epsilon}
    Mod(t) \le \frac{C\Vert \varepsilon \Vert_{H^1}}{a^2}+ \frac{C_N}{a^{N+1}} + C_N\Vert \varepsilon \Vert_{H^1}^2.    
\end{equation}
Using \eqref{eq bootstrap assumption} and \eqref{eq asymptotic behavior of hyperbolic trajectory}, we can get the decay of $Mod$:
\begin{equation} \label{eq estimate of Mod}
    Mod(t) \le t^{-\frac{N}{4}}, \qquad \forall t \in [T_*,T_n].
\end{equation}

We are now going to deduce the second and third lines of \eqref{eq bootstrap conclusion}. By the fundamental theorem of calculus, we have
\begin{equation} \begin{aligned}
    &\big| M_j(P)- M_j(P^{(N)}) \big|+ \big| B_j(P)- B_j(P^{(N)}) \big| \\
    \le &\ C \left( \frac{|\alpha- \alpha^{(N)}|}{a^3}+ \frac{|\beta- \beta^{(N)}|}{a^2}+ \frac{|\lambda- \lambda^{(N)}|}{a^2} \right) \\
    &\ + C_N \left( \frac{|\alpha- \alpha^{(N)}|}{a^4}+ \frac{|\beta- \beta^{(N)}|}{a^3}+ \frac{|\lambda- \lambda^{(N)}|}{a^3} \right).
\end{aligned} \end{equation}
Using \eqref{eq trajectory}, \eqref{eq bootstrap assumption}, \eqref{eq asymptotic behavior of hyperbolic trajectory} and \eqref{eq estimate of Mod}, if  $T_0(N)$ is large enough, then we get
\begin{equation} \begin{aligned}
    &\big| \dot{\lambda}_j- \dot{\lambda}_j^{(N)} \big|+ \big| \dot{\beta}_j- \dot{\beta}_j^{(N)} \big| \\
    \le &\ Mod+ \big| M_j(P)- M_j(P^{(N)}) \big|+ \big| B_j(P)- B_j(P^{(N)}) \big| \le Ct^{-2-\frac{N}{8}}.
\end{aligned} \end{equation}
Integrating in $t$ and using \eqref{eq value of epsilon at T_n}, we deduce
\begin{equation} \label{eq bootstrap conclusion, beta and lambda}
    \big| \lambda_j- \lambda_j^{(N)} \big|+ \big| \beta_j- \beta_j^{(N)} \big| \le \frac{C}{N} t^{-1-\frac{N}{8}} \le \frac{1}{2} t^{-1-\frac{N}{8}}
\end{equation}
when $N$ is large enough. This is the second line of \eqref{eq bootstrap conclusion}.

By \eqref{eq estimate of Mod} and \eqref{eq asymptotic behavior of hyperbolic trajectory}, we also have
\begin{equation}
    \big| \dot{\alpha}_j- \dot{\alpha}_j^{(N)} \big| \le 2 \big| \beta_j- \beta_j^{(N)} \big|+ t^{-\frac{N}{4}} 
\end{equation}
and 
\begin{equation}
    \big| \dot{\gamma}_j- \dot{\gamma}_j^{(N)} \big| \le C \left( \big| \lambda_j- \lambda_j^{(N)} \big|+ \big| \beta_j- \beta_j^{(N)} \big|+ t\big| \dot{\beta}_j- \dot{\beta}_j^{(N)} \big|+ \big| \dot{\alpha}_j- \dot{\alpha}_j^{(N)} \big| \right)+ t^{-\frac{N}{4}}.
\end{equation}
We then deduce the third line of \eqref{eq bootstrap conclusion} for large $N$ by \eqref{eq value of epsilon at T_n} and \eqref{eq bootstrap conclusion, beta and lambda}.

\subsection{Control of the error} \label{subsec error}

In this subsection, we aim at proving the first line of \eqref{eq bootstrap conclusion}. We start with the construction of cutoff functions.

\begin{lem} \label{lem cutoff}
There exist $c,C>0$ and $\varphi_j \in C^{1,\infty} (\mathbb{R}_+ \times \mathbb{R}^3)$ for $1 \le j \le m$ such that
\begin{equation} \label{eq cutoff} \begin{gathered}
    0 \le \varphi_j(t,x) \le 1, \quad \sum_{j=1}^m \varphi_j(t,x) \equiv 1, \\ 
    |\partial_t \varphi_j|+ |\nabla \varphi_j| \le \frac{C}{a}, \quad |\partial_t \sqrt{\varphi_j}|+ |\nabla \sqrt{\varphi_j}| \le \frac{C}{a}, \\
    \varphi_j(t,x)= \left\{ \begin{aligned}
        &1, \quad |x-\alpha_j(t)| \le ca(t), \\
        &0, \quad |x-\alpha_k(t)| \le ca(t),\ k \neq j.
    \end{aligned} \right.
\end{gathered} \end{equation}
\end{lem}

\begin{proof}
There exist $c_0,C_0>0$ such that $c_0 t < a(t) < C_0 t$. Take $c \in (0, \frac{1}{2})$ and $r,R>0$ such that $cC_0<r<R<\frac{1}{2}c_0$. Let $\Phi \in C^\infty(\mathbb{R}^3)$ be such that 
\begin{equation}
    0 \le \Phi \le 1, \quad \text{supp } \Phi \subset B_R, \quad \Phi=1 \text{ in } B_r, \quad |\nabla \Phi| \le C\sqrt{1-\Phi}.
\end{equation}
Here, the last property can be satisfied by taking $\Phi_0$ satisfying the other properties and setting $\Phi= 1- (1-\Phi_0)^2$. Then we define
\begin{equation}
    \varphi_j(t,x)= \Phi^2 \Big( \frac{x-\alpha_j(t)}{t} \Big), \quad j=1,2, \cdots, m-1,
\end{equation}
and
\begin{equation}
    \varphi_m(t,x)= 1- \sum_{j=1}^{m-1} \varphi_j(t,x).
\end{equation}
We claim that \eqref{eq cutoff} holds. Only the estimates on $\varphi_m$ need to be checked. 

We have $\varphi_m \in [0,1]$ because $\text{supp } \varphi_j$ are pairwise disjoint. The derivatives of $\sqrt{\varphi_m}$ are bounded because of the last property of $\Phi$.
\end{proof}

Combining the properties of the cutoff functions and \eqref{eq estimate of admissible functions}, we have
\begin{equation} \label{eq localization of R_j and cutoff}
    |\varphi_j R- R_j| \le C_N e^{-ca(t)}, \quad \forall t>0,\ x \in \mathbb{R}^3.
\end{equation}
This means $\varphi_j$ localizes the multisoliton solutions.

Consider the sum of truncated conserved quantities of the Hartree equation
\begin{equation} \begin{aligned}
    \mathcal{E}(u)= &\int |\nabla u|^2- \frac{1}{2} \int \left| \nabla \phi_{|u|^2} \right|^2 \\
    &+ \sum_{j=1}^m \bigg[ \Big( \frac{1}{\lambda_j^2}+ |\beta_j|^2 \Big) \int \varphi_j |u|^2- 2\beta_j \int \varphi_j \mathrm{Im} (\nabla u \overline{u}) \bigg].
\end{aligned} \end{equation}
By the decomposition $u=R+ \varepsilon$, we can expand $\E$ in terms of $\varepsilon$. Then the second or higher order terms is $\G(\varepsilon)= \G_1+\G_2+\G_3$, where
\begin{gather*}
    \G_1= \int |\nabla \varepsilon|^2+ \int \phi_{|R|^2} |\varepsilon|^2- 2\int |\nabla \phi_{\mathrm{Re} (\varepsilon \overline{R}) }|^2+ 2\int \phi_{\mathrm{Re} (\varepsilon \overline{R}) } |\varepsilon|^2- \frac{1}{2} \int |\nabla \phi_{|\varepsilon|^2} |^2, \\
    \G_2= \sum_{j=1}^m \Big( \frac{1}{\lambda_j^2}+ |\beta_j|^2 \Big) \int \varphi_j |\varepsilon|^2, \qquad \G_3= -2\sum_{j=1}^m \beta_j \int \varphi_j \mathrm{Im} (\nabla \varepsilon \overline{\varepsilon}).
\end{gather*}

We point out that because of the orthogonality condition \eqref{eq orthogonality}, the first order term of $\varepsilon$ would vanish if $R$ solved \eqref{eq ground state}. Unfortunately, as $R$ is an approximate solution, it does not solve \eqref{eq ground state}, and the error is too big so that one cannot proceed in this way. 

Therefore, we will directly prove the following two estimates on $\G$. The first one states the positiveness of $\G$, which follows because of orthogonality \eqref{eq orthogonality}. The second follows by a direct calculation and gives an estimate on the upper bound of $\G$. 

\begin{prop} \label{prop coercivity}
Let $N \ge 2$. For $T_0=T_0(N)$ large enough, there exists $c_0>0$ such that
\begin{equation}
    \G(\varepsilon(t)) \ge c_0 \Vert \varepsilon \Vert_{H^1}^2, \quad \forall t \in [T_*,T_n].
\end{equation}
\end{prop}

\begin{prop} \label{prop estimate on G(epsilon)}
Let $N \ge 2$. For $T_0=T_0(N)$ large enough, if \eqref{eq bootstrap assumption} holds, then there exists $C>0$ such that
\begin{equation}
     \left| \frac{\d}{\dt} \G(\varepsilon(t)) \right| \le C t^{-1- \frac{N}{2}}, \quad \forall t \in [T_*,T_n].
\end{equation}
\end{prop}

If taking these two results for granted, then by \eqref{eq value of epsilon at T_n}, we have
\begin{equation}
    c_0 \Vert \varepsilon \Vert_{H^1}^2 \le |\G(\varepsilon(t))| \le \int_t^{T_n} C \tau^{-1- \frac{N}{2}} \d \tau \le \frac{C}{N} t^{-\frac{N}{2}}.
\end{equation}
Taking $N$ large enough, then we obtain the first line of \eqref{eq bootstrap conclusion}.

The rest of this subsection is for the proof of the two propositions. Once they are proved, we will have completed the proof of the hyperbolic case.

\begin{proof} [Proof of Proposition \ref{prop coercivity}] \ \par
The main ingredient of the proof is the following coercivity result on the linearized operators $L_-$ and $L_+$.

\begin{lem} \label{lem coercivity of L+, L-}
There exist $\delta,c>0$ such that: if $v \in H^1$ is real-valued, then
\begin{equation} \begin{gathered} 
    |(v,Q)|+ |(v,xQ)|< \delta \Vert v \Vert_{H^1} \implies (L_+v,v) \ge c \Vert v \Vert_{H^1}^2, \\
    |(v, \Lambda Q)|< \delta \Vert v \Vert_{H^1} \implies (L_-v,v) \ge c \Vert v \Vert_{H^1}^2.
\end{gathered} \end{equation}
\end{lem}

\begin{proof}
All functions in this proof are assumed to be real-valued.

It suffices to prove that for some $c>0$, if $v \in H^1$, then
\begin{equation} \label{eq coercivity of L+, L-} \begin{gathered}
    (v,Q)= (v,xQ)= 0 \implies (L_+v,v) \ge c \Vert v \Vert_{H^1}^2, \\
    (v, \Lambda Q)= 0 \implies (L_-v,v) \ge c \Vert v \Vert_{H^1}^2.
\end{gathered} \end{equation}
Let us only prove the sufficiency for the estimate on $L_+$. For $v \in H^1$, we have
\begin{equation}
    (L_+v, v)= \int |\nabla v|^2+ \int |v|^2+ \int \phi_{Q^2} v^2- 2\int |\nabla \phi_{Qv}|^2.
\end{equation}
If $u=v-w$, then by the Sobolev inequality, we have
\begin{equation}
     (L_+v,v)- (L_+u,u) \ge -C \Vert u \Vert_{H^1} \Vert w \Vert_{H^1}- C\Vert w \Vert_{H^1}^2. 
\end{equation}
We take
\begin{equation}
    w= \frac{(v,Q)}{\Vert Q \Vert_{L^2}^2} Q+ \frac{(v,xQ)}{\Vert xQ \Vert_{L^2}^2} \cdot xQ.
\end{equation}
Then $\Vert w \Vert_{H^1} \le C\delta \Vert v \Vert_{H^1}$, and $u$ is orthogonal to both $Q$ and $xQ$, so we have $(L_+u,u) \ge c \Vert u \Vert_{H^1}^2$. We thus deduce $(L_+v,v) \ge \frac{c}{2} \Vert v \Vert_{H^1}^2$ by Cauchy-Schwarz and taking $\delta$ small. 

We then turn to prove \eqref{eq coercivity of L+, L-}. We only prove the first line, as the proof of the second line is similar and easier. 

Recall that $Q$ is a minimizer of 
\begin{equation}
    \H(u)= \frac{1}{2} \int |\nabla u|^2- \frac{1}{4} \int |\nabla \phi_{|u|^2} |^2, \quad \text{where } u \in H^1 \text{ and } \Vert u \Vert_{L^2}= \Vert Q \Vert_{L^2}.
\end{equation}
Thus, $Q$ is a minimizer in $H^1 \backslash \{0\}$ of
\begin{equation}
    \J(u):= \frac{\Vert Q \Vert_{L^2}^2}{\Vert u \Vert_{L^2}^2} \int |\nabla u|^2- \frac{\Vert Q \Vert_{L^2}^4}{2\Vert u \Vert_{L^2}^4} \int |\nabla \phi_{|u|^2} |^2.
\end{equation}
Assume $f \in H^1$ and $(f,Q)=0$. By direct computation, we have
\begin{equation} \begin{aligned}
    \frac{1}{2} \frac{\d^2}{\d \epsilon^2} \bigg|_{\epsilon=0} \J(Q+ \epsilon f) = & \int |\nabla f|^2- \frac{\Vert \nabla Q \Vert_{L^2}^2}{\Vert Q \Vert_{L^2}^2} \int f^2 \\
    &- 2\int |\nabla \phi_{Qf}|^2- \int \nabla \phi_{Q^2} \cdot \nabla \phi_{f^2}+ \frac{\Vert \nabla \phi_{Q^2} \Vert_{L^2}^2}{\Vert Q \Vert_{L^2}^2} \int f^2.
\end{aligned} \end{equation}
Using \eqref{eq ground state} and integration by parts, we get
\begin{equation}
    \Vert \nabla Q \Vert_{L^2}^2- \Vert \nabla \phi_{Q^2} \Vert_{L^2}^2+ \Vert Q \Vert_{L^2}^2=0.
\end{equation}
Thus we obtain
\begin{equation}
    \frac{1}{2} \frac{\d^2}{\d \epsilon^2} \bigg|_{\epsilon=0} \J(Q+ \epsilon f) = (L_+f,f).
\end{equation}
By the minimality of $\J(Q)$, we deduce that 
\begin{equation}
    f \in H^1,\ (f,Q)=0 \implies (L_+f,f) \ge 0. 
\end{equation}

Therefore, if the conclusion fails to be true, then there exist $f_n \in H^1$ such that $(f_n,Q)= (f_n,xQ)=0$, $\Vert f_n \Vert_{H^1}=1$ and $(L_+ f_n, f_n) \to 0$. By passing to subsequence, we may assume $f_n \rightharpoonup f_0$ in $H^1$. We have $(f_0,Q)= (f_0,xQ)=0$, and by the Rellich-Kondrachov theorem and the decay of $Q$, we have
\begin{equation}
    \int \phi_{Q^2} f_n^2 \to \int \phi_{Q^2} f_0^2 \quad \text{and} \quad \int |\nabla \phi_{Qf_n}|^2 \to \int |\nabla \phi_{Qf_0}|^2.
\end{equation}
On the other hand, $\Vert f_0 \Vert_{H^1} \le \liminf \Vert f_n \Vert_{H^1}$. We deduce that $(L_+f_0, f_0) \le 0$.

By the non-negativity of $L^+$ on $Q^\perp$, we know that $f_0$ is a nonzero minimizer of $(L_+u,u)$, where $u \in H^1$ and $(u,Q)=0$. By computing the first variation, we get
\begin{equation}
    L_+f_0= aQ \quad \text{for some } a \in \mathbb{R}.
\end{equation}
Since $L_+(\Lambda Q)= -2Q$ and $\ker L_+$ is spanned by $\nabla Q$ (Theorem 4 in \cite{Lenzmanngroundstate}), we know $f_0$ is a linear combination of $\Lambda Q$ and $\nabla Q$. But using $(f_0,Q)= (f_0,xQ)=0$, we must have $f_0=0$, which is a contradiction.
\end{proof}

Let $\varepsilon_j= \varepsilon \sqrt{\varphi_j}$ and $\tilde{\varepsilon}_j= g_j^{-1} \varepsilon_j$, or more precisely, define
\begin{equation}
    \tilde{\varepsilon}_j(t,y_j)= \lambda_j^2(t) \varepsilon_j(t, \lambda_j(t)y_j+ \alpha_j(t)) e^{i\gamma_j(t)- i\beta_j(t) \cdot (\lambda_j(t)y_j+ \alpha_j(t))}.
\end{equation}
By \eqref{eq orthogonality}, \eqref{eq estimate of admissible functions} and \eqref{eq asymptotic behavior of hyperbolic trajectory}, for $t$ large enough, we can apply Lemma \ref{lem coercivity of L+, L-} to $\mathrm{Re}( \tilde{\varepsilon}_j)$ and $\mathrm{Im}( \tilde{\varepsilon}_j)$. Thus for $T_0(N)$ large enough, we have
\begin{equation}
    \Big( L_+ \mathrm{Re}(\tilde{\varepsilon}_j), \mathrm{Re}(\tilde{\varepsilon}_j) \Big)+ \Big( L_- \mathrm{Im}(\tilde{\varepsilon}_j), \mathrm{Im}(\tilde{\varepsilon}_j) \Big) \ge c \Vert \tilde{\varepsilon}_j \Vert_{H^1}^2, \quad \forall t\ge T_0.
\end{equation}
In the following, we always assume $T_0$ is large enough so that the above holds.

Set $Q_j= g_j Q$. By computation similar to \eqref{eq basic calculation} and \eqref{eq estimate of admissible functions}, we have
\begin{equation} \begin{aligned}
    &\Big( L_+ \mathrm{Re}(\tilde{\varepsilon}_j), \mathrm{Re}(\tilde{\varepsilon}_j) \Big)+ \Big( L_- \mathrm{Im}(\tilde{\varepsilon}_j), \mathrm{Im}(\tilde{\varepsilon}_j) \Big) \\
    = &\ \int |\nabla \tilde{\varepsilon}_j|^2+ \int |\tilde{\varepsilon}_j|^2+ \int \phi_{Q^2} |\tilde{\varepsilon}_j|^2+ 2\int \phi_{\mathrm{Re}(Q \tilde{\varepsilon}_j)} \mathrm{Re}(Q \tilde{\varepsilon}_j) \\
    = &\ \lambda_j^3 \left( \int |\nabla \varepsilon_j|^2- 2\beta_j \int \mathrm{Im} (\nabla \varepsilon_j \overline{\varepsilon_j})+ |\beta_j|^2 \int |\varepsilon_j|^2 \right)+ \lambda_j \int |\varepsilon_j|^2 \\
    &+ \lambda_j^3 \int \phi_{Q_j^2} |\varepsilon_j|^2+ 2\lambda_j^3 \int \phi_{\mathrm{Re}(\varepsilon_j \overline{Q_j})} \mathrm{Re} (\varepsilon \overline{Q_j}) \\
    = &\ \lambda_j^3 \Bigg( \int |\nabla \varepsilon_j|^2+ \int \phi_{|R_j|^2} |\varepsilon_j|^2- 2\int \big| \nabla \phi_{\mathrm{Re}(\varepsilon_j \overline{{R_j}})} \big|^2 \\
    &\qquad \ + \Big( \frac{1}{\lambda_j^2}+ |\beta_j|^2 \Big) \int |\varepsilon_j|^2- 2\beta_j \int \mathrm{Im} (\nabla \varepsilon_j \overline{\varepsilon_j}) \Bigg)+ O_N \left( \frac{\Vert \varepsilon_j \Vert_{H^1}^2}{a} \right)
\end{aligned} \end{equation}
We then deduce that 
\begin{equation} \label{eq coercivity of H_j} 
    \H_j(\varepsilon_j) \ge c \Vert \varepsilon_j \Vert_{H^1}- \frac{C_N}{a} \Vert \varepsilon_j \Vert_{H^1}^2 \ge c\int \varphi_j ( |\nabla \varepsilon|^2+ |\varepsilon|^2)- \frac{C_N}{a} \Vert \varepsilon \Vert_{H^1}^2,
\end{equation}
where
\begin{equation} \begin{aligned}
    \H_j(\varepsilon_j) = &\ \int |\nabla \varepsilon_j|^2+ \int \phi_{|R_j|^2} |\varepsilon_j|^2- 2\int \big| \nabla \phi_{\mathrm{Re}(\varepsilon_j \overline{{R_j}})} \big|^2 \\
    &\ + \Big( \frac{1}{\lambda_j^2}+ |\beta_j|^2 \Big) \int |\varepsilon_j|^2- 2\beta_j \int \mathrm{Im} (\nabla \varepsilon_j \overline{\varepsilon_j}).
\end{aligned} \end{equation}

Next, we consider the truncated functional
\begin{equation} \begin{aligned}
    \H_{j,\varphi}(\varepsilon)= &\int \varphi_j |\nabla \varepsilon|^2+ \int \phi_{|R_j|^2} |\varepsilon|^2- 2\int \big| \nabla \phi_{\mathrm{Re}(\varepsilon \overline{{R_j}})} \big|^2 \\
    &\ + \Big( \frac{1}{\lambda_j^2}+ |\beta_j|^2 \Big) \int \varphi_j |\varepsilon|^2- 2\beta_j \int \varphi_j \mathrm{Im} (\nabla \varepsilon \overline{\varepsilon}).
\end{aligned} \end{equation}
By \eqref{eq cutoff}, \eqref{eq localization of R_j and cutoff} and \eqref{eq coercivity of H_j}, we have
\begin{equation} \label{eq coercivity of H_j,phi} \begin{aligned}
    \H_{j,\varphi}(\varepsilon)= &\int |\nabla \varepsilon_j|^2+ \int \phi_{|R_j|^2} |\varepsilon|^2- 2\int \big| \nabla \phi_{\mathrm{Re}(\varepsilon \overline{{R_j}})} \big|^2 \\
    &\ + \Big( \frac{1}{\lambda_j^2}+ |\beta_j|^2 \Big) \int |\varepsilon_j|^2- 2\beta_j \int \mathrm{Im} (\nabla \varepsilon_j \overline{\varepsilon_j})+ O \left( \frac{\Vert \varepsilon \Vert_{H^1}^2}{a} \right) \\
    = &\ \H_j(\varepsilon_j)+ \int (1-\sqrt{\varphi_j}) \phi_{|R_j|^2} |\varepsilon|^2- 2\int \big| \nabla \phi_{\mathrm{Re} \big( (1-\sqrt{\varphi_j}) \varepsilon R_j \big)} \big|^2 \\
    &\ + 4\int \phi_{\mathrm{Re}( \varepsilon_j R_j )} \mathrm{Re} \big( (1-\sqrt{\varphi_j}) \varepsilon R_j \big)+ O \left( \frac{\Vert \varepsilon \Vert_{H^1}^2}{a} \right) \\
    \ge &\ c\int \varphi_j ( |\nabla \varepsilon|^2+ |\varepsilon|^2)- \frac{C_N}{a} \Vert \varepsilon \Vert_{H^1}^2.
\end{aligned} \end{equation}

Finally, using the sum-to-1 property of $\varphi_j$, we write
\begin{equation} \begin{aligned}
    \G(\varepsilon)= &\ \sum_{j=1}^m \H_{j,\varphi}(\varepsilon)+ 2\int \phi_{\mathrm{Re} (\varepsilon \overline{R})} |\varepsilon|^2- \frac{1}{2} \int |\nabla \phi_{|\varepsilon|^2}|^2 \\
    &\ + \sum_{j \neq k} \int \phi_{\mathrm{Re} (R_k \overline{R_j})} |\varepsilon|^2- 2\sum_{j \neq k} \int \nabla \phi_{\mathrm{Re} (\varepsilon \overline{R_j})} \cdot \nabla \phi_{\mathrm{Re} (\varepsilon \overline{R_k})}. 
\end{aligned} \end{equation}
The first term is controlled by \eqref{eq coercivity of H_j,phi}. The other terms in the first line are $O \big( t^{-N/4} \Vert \varepsilon \Vert_{H^1}^2 \big)$ because of \eqref{eq bootstrap assumption}. Using Lemma \ref{lem localization}, we know the two terms in the second line are $O_N \big( e^{-ca}\Vert \varepsilon \Vert_{H^1}^2 \big)$ and $O_N \Big( \frac{\Vert \varepsilon \Vert_{H^1}^2}{a} \Big)$, respectively. We thus have
\begin{equation}
    \G(\varepsilon) \ge c \Vert \varepsilon \Vert_{H^1}^2- \frac{C_N}{a} \Vert \varepsilon \Vert_{H^1}^2.
\end{equation}
Thanks to \eqref{eq asymptotic behavior of hyperbolic trajectory}, we conclude by taking $T_0(N)$ large enough.
\end{proof}

\begin{proof} [Proof of Proposition \ref{prop estimate on G(epsilon)}]
We deal with $\G_1$, $\G_2$ and $\G_3$ separately.

(1) Using integration by parts, we have
\begin{equation} \begin{aligned}
    \frac{\d \G_1}{\dt}= &\ -2\mathrm{Im} \int i\partial_t \overline{\varepsilon} \left( \Delta \varepsilon- \phi_{|R|^2} \varepsilon- 2\phi_{\mathrm{Re} (\varepsilon \overline{R})} R- \N(\varepsilon) \right) \\
    &\ +4\mathrm{Re} \int \phi_{\mathrm{Re} (\varepsilon \overline{R})} \varepsilon \partial_t \overline{R}+ 2\int \phi_{\mathrm{Re} (\partial_t R \overline{R})} |\varepsilon|^2+ 2\mathrm{Re} \int \phi_{|\varepsilon|^2} \varepsilon \partial_t \overline{R}.
\end{aligned} \end{equation}
By \eqref{eq equation of epsilon}, \eqref{eq estimate of Psi}, \eqref{eq estimate of S_j with Mod} and \eqref{eq estimate of N(epsilon)}, the first line is $O_N \left( \big( \frac{1}{a^{N+1}}+ Mod \big) \Vert \varepsilon \Vert_{H^1} \right)$. For the second line, using \eqref{eq basic calculation}, \eqref{eq asymptotic behavior of hyperbolic trajectory} and \eqref{eq estimate of M_j, B_j}, we have
\begin{equation} \begin{aligned}
    \partial_t R_j &= \frac{1}{\lambda_j^2} \left( -\frac{\dot{\alpha}_j}{\lambda_j} \cdot \nabla V_j- i(\dot{\gamma}_j- \dot{\beta}_j \cdot x) V_j \right) e^{-i\gamma_j+ i\beta \cdot x}+ O_N \left( \frac{1}{a^2} \right) e^{-c_N|x-\alpha_j|} \\
    &= -2\beta_j \cdot \nabla R_j+ i\big( \frac{1}{\lambda_j^2}+ |\beta_j|^2 \big) R_j+ O_N \left( \frac{1}{a^2}+ Mod \right) e^{-c_N|x-\alpha_j|}.
\end{aligned} \end{equation}
Combining these with \eqref{eq estimate of Mod with epsilon}, \eqref{eq decay of R_j} and Lemma \ref{lem localization}, we get
\begin{equation} \label{eq estimate of G_1} \begin{aligned}
    \frac{\d \G_1}{\dt}= \sum_{j=1}^m \Bigg( 4\Big( & \frac{1}{\lambda_j^2}+ |\beta_j|^2 \Big) \int \phi_{\mathrm{Re}( \varepsilon \overline{R})} \mathrm{Im} (\varepsilon \overline{R_j})- 8\int \phi_{\mathrm{Re} (\varepsilon \overline{R})} \mathrm{Re} (\varepsilon \beta_j \cdot \nabla \overline{R_j}) \\
    &- 4\int \phi_{\mathrm{Re} (\beta_j \cdot \nabla R_j \overline{R_j})} |\varepsilon|^2 \Bigg)+ O_N \left( \frac{\Vert \varepsilon \Vert_{H^1}}{a^{N+1}}+ \frac{\Vert \varepsilon \Vert_{H^1}^2} {a^2}+ \Vert \varepsilon \Vert_{H^1}^3 \right).
\end{aligned} \end{equation} 

(2) Using \eqref{eq asymptotic behavior of hyperbolic trajectory} and \eqref{eq cutoff}, we have
\begin{equation}
    \frac{\d \G_2}{\dt}= \sum_{j=1}^m 2\Big( \frac{1}{\lambda_j^2}+ |\beta_j|^2 \Big) \int \varphi_j \mathrm{Im} (i\partial_t \varepsilon \overline{\varepsilon})+ O\left( \frac{\Vert \varepsilon \Vert_{H^1}^2}{t} \right). 
\end{equation}
Then by \eqref{eq equation of epsilon}, \eqref{eq estimate of Psi}, \eqref{eq estimate of S_j with Mod} and \eqref{eq estimate of N(epsilon)}, we have
\begin{equation} \begin{aligned}
    \frac{\d \G_2}{\dt} &= \sum_{j=1}^m 2\Big( \frac{1}{\lambda_j^2}+ |\beta_j|^2 \Big) \int \varphi_j \mathrm{Im} (2\phi_{\mathrm{Re}(\varepsilon \overline{R})} R \overline{\varepsilon})+ O\left( \frac{\Vert \varepsilon \Vert_{H^1}^2}{t} \right) \\
    &\quad + O_N \left( \frac{\Vert \varepsilon \Vert_{H^1}}{a^{N+1}}+ Mod \Vert \varepsilon \Vert_{H^1}+ \Vert \varepsilon \Vert_{H^1}^3 \right).
\end{aligned} \end{equation}
Finally using \eqref{eq estimate of Mod with epsilon} and \eqref{eq localization of R_j and cutoff}, we get
\begin{equation} \label{eq estimate of G_2} \begin{aligned}
    \frac{\d \G_2}{\dt} &= \sum_{j=1}^m- 4\Big( \frac{1}{\lambda_j^2}+ |\beta_j|^2 \Big) \int \phi_{\mathrm{Re}(\varepsilon \overline{R})} \mathrm{Im} (\varepsilon \overline{R_j})+ O\left( \frac{\Vert \varepsilon \Vert_{H^1}^2}{t} \right) \\
    &\quad + O_N \left( \frac{\Vert \varepsilon \Vert_{H^1}}{a^{N+1}}+ \frac{\Vert \varepsilon \Vert_{H^1}^2} {a^2}+ \Vert \varepsilon \Vert_{H^1}^3 \right).
\end{aligned} \end{equation} 

(3) Similarly, we can compute
\begin{equation} \label{eq estimate of G_3} \begin{aligned}
    \frac{\d \G_3}{\dt} &= \sum_{j=1}^m -4\beta_j \int \varphi_j \mathrm{Re}(i\partial_t \varepsilon \nabla \overline{\varepsilon})+ O\left( \frac{\Vert \varepsilon \Vert_{H^1}^2}{t} \right) \\
    &= \sum_{j=1}^m -4\beta_j \int \varphi_j \mathrm{Re} \left( 2\phi_{\mathrm{Re} (\varepsilon \overline{R})} R \nabla \overline{\varepsilon}+ \phi_{|R|^2} \varepsilon \nabla \overline{\varepsilon} \right)+ O\left( \frac{\Vert \varepsilon \Vert_{H^1}^2}{t} \right) \\
    &\quad + O_N \left( \frac{\Vert \varepsilon \Vert_{H^1}}{a^{N+1}}+ Mod \Vert \varepsilon \Vert_{H^1}+ \Vert \varepsilon \Vert_{H^1}^3 \right) \\
    &= \sum_{j=1}^m \left( 8\int \phi_{\mathrm{Re} (\varepsilon \overline{R})} \mathrm{Re}(\varepsilon \beta_j \varphi_j \cdot \nabla \overline{R})+ 4\int \phi_{\mathrm{Re} (\beta_j \varphi_j \cdot \nabla R \overline{R})} |\varepsilon|^2 \right) \\
    &\quad +O\left( \frac{\Vert \varepsilon \Vert_{H^1}^2}{t} \right)+ O_N \left( \frac{\Vert \varepsilon \Vert_{H^1}}{a^{N+1}}+ Mod \Vert \varepsilon \Vert_{H^1}+ \Vert \varepsilon \Vert_{H^1}^3 \right) \\
    &= \sum_{j=1}^m \left( 8\int \phi_{\mathrm{Re} (\varepsilon \overline{R})} \mathrm{Re}(\varepsilon \beta_j \cdot \nabla \overline{R_j})+ 4\int \phi_{\mathrm{Re} (\beta_j \cdot \nabla R_j \overline{R_j})} |\varepsilon|^2 \right) \\
    &\quad +O\left( \frac{\Vert \varepsilon \Vert_{H^1}^2}{t} \right)+ O_N \left( \frac{\Vert \varepsilon \Vert_{H^1}}{a^{N+1}}+ \frac{\Vert \varepsilon \Vert_{H^1}^2} {a^2}+ \Vert \varepsilon \Vert_{H^1}^3 \right).
\end{aligned} \end{equation}

Combining \eqref{eq estimate of G_1}, \eqref{eq estimate of G_2}, \eqref{eq estimate of G_3} together, we deduce
\begin{equation}
    \left| \frac{\d}{\dt} \G(\varepsilon(t)) \right| \le \frac{C\Vert \varepsilon \Vert_{H^1}^2}{t}+ C_N \left( \frac{\Vert \varepsilon \Vert_{H^1}}{a^{N+1}}+ \frac{\Vert \varepsilon \Vert_{H^1}^2} {a^2}+ \Vert \varepsilon \Vert_{H^1}^3 \right)
\end{equation}
Using \eqref{eq asymptotic behavior of hyperbolic trajectory}, \eqref{eq bootstrap assumption} and taking $T_0(N)$, $N$ large enough, we get the desired result.
\end{proof}

We have finished the proof of the hyperbolic case.

\section{The parabolic and the hyperbolic-parabolic case} \label{sec parabolic and hyperbolic-parabolic}

One of the difficulties of dealing with these two cases is to establish Proposition \ref{prop hyperbolic trajectory}. Due to the lower rates of expansion, we need more delicate computation.

\subsection{The approximate trajectory}

The goal of this subsection is to prove the following analog of Proposition \ref{prop hyperbolic trajectory}.

\begin{prop} \label{prop parabolic and hyperbolic-parabolic trajectory}
Let $P^{\infty}$ be a parabolic or hyperbolic-parabolic solution to \eqref{eq m-body problem} of the form \eqref{eq parabolic solution} or \eqref{eq hyperbolic-parabolic solution}, and $B_j^{(N)}, M_j^{(N)}$ be as in Proposition \ref{prop construction of approximate bubbles}. Then for $N \ge 3$, $\exists T_0= T_0(N)>0$ and $P^{(N)} \in C^1 \big( [T_0,+\infty), \Omega \big)$ satisfying \eqref{eq trajectory} and for any $t \ge T_0$,
\begin{equation} \label{eq estimate of parabolic trajectory}
    |\alpha_j^{(N)}(t)- \alpha_j^\infty(t)| \le t^{-\frac{1}{4}}, \quad |\beta_j^{(N)}(t)- \beta_j^\infty(t)| + |\lambda_j^{(N)}(t)- \lambda_j^\infty| \le t^{-\frac{1}{2}}.
\end{equation}
\end{prop}

\begin{rmk}
This proposition is stronger than the one in \cite{KMR2bodyHartree} because: (1) we do not need to assume $\lambda_j^\infty$ are identical in the parabolic case, or $\lambda_j^\infty$ are identical for $j \in J$ in the hyperbolic-parabolic case; (2) we know $\alpha_j^{(N)}(t)- \alpha_j^\infty(t) \to 0$ as $t \to +\infty$.
\end{rmk}

While before we have only used the formula of $b_j^{(2)}$, here we also need the explicit expression of $m_j^{(2)}$ and $b_j^{(3)}$. Since $T_j^{(1)}$ is real-valued, by \eqref{eq definition of E_j tilde}, we have
\begin{equation}
    \mathrm{Im} \ \hat{E}_j^{(1)} = \lambda_j^2 \sum_{k=1}^m \frac{\partial T_j^{(1)}}{\partial \alpha_k} \cdot 2\beta_k.
\end{equation}
By the formula of $T_j^{(1)}$ \eqref{eq formula of T_j^1} and the requirement of $m_j^{(2)}$, we deduce
\begin{equation} \label{eq formula of m_j^2}
    m_j^{(2)}(P)= \sum_{k \neq j} \frac{\lambda_j^3 \Vert Q \Vert_{L^2}^2}{4\pi \lambda_k} \frac{\alpha_{jk} \cdot \beta_{jk}}{|\alpha_{jk}|^3}.
\end{equation}
Combining this and calculation in the proof of Proposition \ref{prop construction of approximate bubbles}, we have
\begin{equation} 
    \left\{ \begin{aligned}
        &L_+ \mathrm{Re}\ T_j^{(2)}= -2\phi_{Q T_j^{(1)}} T_j^{(1)}- \phi_{ \left| T_j^{(1)} \right|^2} Q- \sum_{k \neq j} \left( 2\psi_{Q T_k^{(1)}}^{(1)} Q+ \psi_{Q^2,k}^{(1)} T_j^{(1)} \right), \\
        &L_- \mathrm{Im}\ T_j^{(2)}= 0.
    \end{aligned} \right.
\end{equation}
Thus we may take $T_j^{(2)}$ as a real-valued radial function. Then we compute
\begin{equation} \begin{aligned}
    \mathrm{Re} \ \hat{E}_j^{(2)}= &\ -\lambda_j^3 b_j^{(2)} \cdot y_j T_j^{(1)}- 2\phi_{T_j^{(1)} T_j^{(2)}} Q- 2\phi_{Q T_j^{(2)}} T_j^{(1)}- 2\phi_{Q T_j^{(1)}} T_j^{(2)} \\
    &\ -\phi_{\left| T_j^{(1)} \right|^2} T_j^{(1)}- \sum_{k \neq j} \bigg( \psi_{Q^2,k}^{(2)} T_j^{(1)}+ 2\psi_{Q T_k^{(1)}}^{(2)} Q+ 2\psi_{Q T_k^{(2)}}^{(1)} Q \\
    &\qquad \qquad \qquad \qquad \qquad + \psi_{\left| T_j^{(1)} \right|^2}^{(1)} Q+ 2\psi_{Q T_k^{(1)}}^{(1)} T_j^{(1)}+ \psi_{Q^2,k}^{(1)} T_j^{(2)} \bigg).
\end{aligned} \end{equation}

Recall that we require $b_j^{(3)}$ to satisfy
\begin{equation}
    \Big( \lambda_j^3 b_j^{(3)} \cdot y_j Q+ \sum_{k \neq j} \psi_{Q^2,k}^{(3)} Q- \mathrm{Re} \ \hat{E}_j^{(2)}, \nabla Q \Big)=0.
\end{equation}
Since $Q, T_j^{(1)}, T_j^{(2)}$ are all even, the terms with $\phi, \psi^{(1)}, \psi^{(3)}$ are orthogonal to $\nabla Q$. Removing those terms and using \eqref{eq formula of b_j^2}, we obtain
\begin{equation}
    \Big( \lambda_j^3 b_j^{(3)} \cdot y_j Q+ 2\sum_{k \neq j} \psi_{Q T_k^{(1)}}^{(2)} Q, \nabla Q \Big)=0.
\end{equation}
Then by \eqref{eq formula of T_j^1} and $(\Lambda Q,Q)= \frac{1}{2} \Vert Q \Vert_{L^2}^2$, a result of integration by parts, we get
\begin{equation} \label{eq formula of b_j^3}
    b_j^{(3)}(P)= - \sum_{k \neq j} \frac{(Q,T_k^{(1)}) \alpha_{jk}}{2\pi \lambda_k |\alpha_{jk}|^3}= \frac{\Vert Q \Vert_{L^2}^4}{32\pi} \sum_{k \neq j} \bigg( \sum_{l \neq k} \frac{1}{\lambda_l |\alpha_{kl}|} \bigg) \frac{\lambda_k \alpha_{jk}}{ |\alpha_{jk}|^3}.
\end{equation}

Then we give the idea of the proof to make it easier to understand.

In the hyperbolic case, the equation can be roughly written as
\begin{equation}
    \left\{ \begin{aligned}
        \dot{\alpha}(t)- \dot{\alpha}^\infty(t) &= 2\beta(t)- 2\beta^\infty(t), \\
        \dot{\beta}(t)- \dot{\beta}^\infty(t) &= O \big( t^{-3} \big), \\
        \dot{\lambda}(t)- \dot{\lambda}_j^\infty(t) &= O \big( t^{-2} \big),
    \end{aligned} \right.  
\end{equation}
and then we can apply the fixed point theorem. But in the parabolic and hyperbolic-parabolic cases, we only have
\begin{equation}
    \left\{ \begin{aligned}
        \dot{\alpha}(t)- \dot{\alpha}^\infty(t) &= 2\beta(t)- 2\beta^\infty(t), \\
        \dot{\beta}(t)- \dot{\beta}^\infty(t) &= O \big( t^{-2} \big), \\
        \dot{\lambda}(t)- \dot{\lambda}_j^\infty(t) &= O \big( t^{-\frac{4}{3}} \big).
    \end{aligned} \right.  
\end{equation}
The error term is so large that the fixed point theorem becomes invalid.

The recipe is to replace $P^\infty$ by some $\tilde{P}$ which is closed to $P^\infty$ and makes the error terms on the right hand side smaller. More precisely, $P^\infty$ eliminates $B_j^{(N)}$ up to the second term and $M_j^{(N)}$ up to the first term. We want $\tilde{P}$ to eliminate $B_j^{(N)}$ up to the third term and $M_j^{(N)}$ up to the second term. Our $\tilde{P}$ serves as $P^{(app)}$ in \cite{KMR2bodyHartree}, but we would like to point out that the existence of such $\tilde{P}$ is not taken for granted when $m \ge 3$. We made a new observation that several terms will cancel. Thanks to this observation, we have an explicit expression of $\tilde{P}$ and more importantly, we know $\tilde{\alpha}= \alpha^\infty$ and $\tilde{\beta}= \beta^\infty$. This is exactly the reason why we can make an improvement.

The approximate equation we will get is \eqref{eq trajectory, modified twice}, which is still more complicated than that in the hyperbolic case. We need a more involved application of the fixed point theorem to obtain the conclusion of the proposition.

\begin{proof} [Proof of Proposition \ref{prop parabolic and hyperbolic-parabolic trajectory}] \

Take $\tilde{\alpha}= \alpha^\infty$, $\tilde{\beta}= \beta^\infty$ and for each $j$, 
\begin{equation} \label{eq formula of tilde lambda}
    \tilde{\lambda}_j(t)= \lambda_j^\infty- \int_t^\infty \sum_{k \neq j} \frac{(\lambda_j^\infty)^3 \Vert Q \Vert_{L^2}^2}{4\pi \lambda_k^\infty} \frac{\alpha_{jk}^\infty (\tau) \cdot \beta_{jk}^\infty (\tau)}{|\alpha_{jk}^\infty (\tau)|^3} \d \tau.
\end{equation}
Using $\dot{\alpha}_j^\infty= 2\beta_j^\infty$, we may simplify the expression as
\begin{equation} \label{eq formula of tilde lambda 2}
    \tilde{\lambda}_j(t)= \lambda_j^\infty- \sum_{k \neq j} \frac{(\lambda_j^\infty)^3 \Vert Q \Vert_{L^2}^2} {8\pi \lambda_k^\infty |\alpha_{jk}^\infty(t)|}.
\end{equation}
In other words, we take $\tilde{P}(t)= \big( \alpha^\infty(t), \beta^\infty(t), \tilde{\lambda}(t) \big)$.

Let $\epsilon= \frac{1}{100}$ and $T_0>0$. Define the norm $\Vert \cdot \Vert_2$ of $P \in C \big( [T_0,+\infty), \Omega \big)$ by
\begin{equation}
    \Vert P \Vert_2:= \sum_{j=1}^m \sup_{t \ge T_0} \Big( t^{\frac{1}{3}- 3\epsilon} |\alpha_j(t)|+ t^{\frac{4}{3}- 2\epsilon} |\beta_j(t)|+t^{1-\epsilon} |\lambda_j(t)| \Big),
\end{equation}
and let $Y= \Big\{ P \in C \big( [T_0,+\infty), \Omega \big) \ \big| \ \Vert P-\tilde{P} \Vert_2 \le 1 \Big\}$. Then it suffices to find a solution of \eqref{eq trajectory} in $Y$. We will assume $P \in Y$ hereinafter.

Recalling the expression \eqref{eq formula of b_j^2}, \eqref{eq formula of m_j^2} and \eqref{eq formula of b_j^3}, if we set
\begin{equation} \begin{aligned}
    \tilde{b}_j^{(2)}(P) &= b_j^{(2)}(P)+ \sum_{k \neq j} \frac{\Vert Q \Vert_{L^2}^2 \alpha_{jk}}{4\pi \lambda_k^\infty |\alpha_{jk}|^3}- \sum_{k \neq j} \frac{\Vert Q \Vert_{L^2}^2 (\lambda_k- \lambda_k^\infty) \alpha_{jk}}{4\pi (\lambda_k^\infty)^2 |\alpha_{jk}|^3}, \\
    \tilde{b}_j^{(3)}(P)&= b_j^{(3)}(P)- \frac{\Vert Q \Vert_{L^2}^4}{32\pi} \sum_{k \neq j} \bigg( \sum_{l \neq k} \frac{1}{\lambda_l^\infty |\alpha_{kl}|} \bigg) \frac{\lambda_k^\infty \alpha_{jk}}{ |\alpha_{jk}|^3}, \\
    \tilde{m}_j^{(2)}(P) &= m_j^{(2)}(P)- \sum_{k \neq j} \frac{(\lambda_j^\infty)^3 \Vert Q \Vert_{L^2}^2}{4\pi \lambda_k^\infty} \frac{\alpha_{jk} \cdot \beta_{jk}}{|\alpha_{jk}|^3},
\end{aligned} \end{equation}
then using $\tilde{a}(t) \sim t^{\frac{2}{3}}$ and $\Vert \tilde{P}- P^\infty \Vert \le Ct^{-\frac{2}{3}}$, we get
\begin{equation} \label{eq estimate of tilde b,m}
    \Big| \tilde{b}_j^{(2)}(\tilde{P}(t)) \Big|+ \Big| \tilde{b}_j^{(3)}(\tilde{P}(t)) \Big| \le Ct^{-\frac{8}{3}}, \quad \Big| \tilde{m}_j^{(2)}(\tilde{P}(t)) \Big| \le Ct^{-2}.
\end{equation}
Moreover, by \eqref{eq formula of tilde lambda} and \eqref{eq formula of tilde lambda 2}, direct computation yields
\begin{equation}
    \dot{\tilde{\beta}}_j= b_j^{(2)}(\tilde{P})- \tilde{b}_j^{(2)}(\tilde{P})+ b_j^{(3)}(\tilde{P})- \tilde{b}_j^{(3)}(\tilde{P}), \quad \dot{\tilde{\lambda}}_j= m_j^{(2)}(\tilde{P})- \tilde{m}_j^{(2)}(\tilde{P}).
\end{equation}
This is the cancellation of errors we have mentioned. Both \eqref{eq formula of tilde lambda} and \eqref{eq formula of tilde lambda 2} are expressions of $\tilde{\lambda}_j$, and they lead to the above formulas of $\dot{\tilde{\lambda}}_j$ and $\dot{\tilde{\beta}}_j$, respectively.

Write $P^{(N)}=P$ for simplicity. Then we can rewrite \eqref{eq trajectory} as
\begin{equation}
    \left\{ \begin{aligned}
        \dot{\alpha}_j(t)- \dot{\tilde{\alpha}}_j(t) &= 2\beta_j(t)- 2\tilde{\beta}_j(t), \\
        \dot{\beta}_j(t)- \dot{\tilde{\beta}}_j(t) &= \left[ b_j^{(2)}(P(t))- b_j^{(2)}(\tilde{P}(t)) \right]+ \left[ b_j^{(3)}(P(t))- b_j^{(3)}(\tilde{P}(t)) \right] \\
        &\quad + \tilde{b}_j^{(2)}(\tilde{P}(t))+ \tilde{b}_j^{(3)}(\tilde{P}(t))+ \sum_{n=4}^N b_j^{(n)}(P(t)), \\
        \dot{\lambda}_j(t)- \dot{\tilde{\lambda}}_j(t) &= \left[ m_j^{(2)}(P(t))- m_j^{(2)}(\tilde{P}(t)) \right]+ \tilde{m}_j^{(2)}(\tilde{P}(t))+ \sum_{n=3}^N m_j^{(n)}(P(t)).
    \end{aligned} \right.
\end{equation}
By \eqref{eq estimate of tilde b,m}, estimates of $b_j^{(n)}$, $m_j^{(n)}$ and $a(t) \sim t^{\frac{2}{3}}$, we have
\begin{equation} \label{eq trajectory, modified}
    \left\{ \begin{aligned}
        \dot{\alpha}_j(t)- \dot{\tilde{\alpha}}_j(t) &= 2\beta_j(t)- 2\tilde{\beta}_j(t), \\
        \dot{\beta}_j(t)- \dot{\tilde{\beta}}_j(t) &= b_j^{(2)}(P(t))- b_j^{(2)}(\tilde{P}(t))+ O \big( t^{-\frac{7}{3}} \big), \\
        \dot{\lambda}_j(t)- \dot{\tilde{\lambda}}_j(t) &= O \big( t^{-2} \big).
    \end{aligned} \right.
\end{equation}
In this proof, $O(t^{-\kappa})$ represents a continuous function of $\tilde{P}$ and $P$, whose $C^1$ norm in $P$ is bounded by $Ct^{-\kappa}$ when evaluated at $(\tilde{P}(t), P(t))$. 

We still need to estimate $b_j^{(2)}(P(t))- b_j^{(2)}(\tilde{P}(t))$. We have
\begin{equation}
    b_j^{(2)}(P)- b_j^{(2)}(\tilde{P})= \frac{\Vert Q \Vert_{L^2}^2}{4\pi} \sum_{k \neq j} \bigg[ \frac{1}{\tilde{\lambda}_k} \Big( \frac{\tilde{\alpha}_{jk}}{|\tilde{\alpha}_{jk}|^3}- \frac{\alpha_{jk}}{|\alpha_{jk}|^3} \Big)+ \frac{(\lambda_k- \tilde{\lambda}_k) \alpha_{jk}} {\lambda_k \tilde{\lambda}_k |\alpha_{jk}|^3} \bigg]. 
\end{equation}
By the Taylor formula,
\begin{equation}
    \frac{\alpha_{jk}}{|\alpha_{jk}|^3}- \frac{\tilde{\alpha}_{jk}}{|\tilde{\alpha}_{jk}|^3} = \frac{\alpha_{jk}- \tilde{\alpha}_{jk}} {|\tilde{\alpha}_{jk}|^3}- \frac{3 \tilde{\alpha}_{jk} \cdot (\alpha_{jk}- \tilde{\alpha}_{jk})} {|\tilde{\alpha}_{jk}|^5} \tilde{\alpha}_{jk}+ O \big( t^{-\frac{8}{3}+ 6\epsilon} \big).
\end{equation}
Using \eqref{eq parabolic solution} and \eqref{eq hyperbolic-parabolic solution}, there exists a matrix $A_{jk} \in \mathbb{R}^{3 \times 3}$ such that
\begin{equation}
    \frac{\alpha_{jk}}{|\alpha_{jk}|^3}- \frac{\tilde{\alpha}_{jk}}{|\tilde{\alpha}_{jk}|^3} = \frac{A_{jk}}{t^2} (\alpha_{jk}- \tilde{\alpha}_{jk})+ O \big( t^{-\frac{7}{3}+ \frac{\epsilon}{2}} \big),
\end{equation}
so there exists $A_j \in \mathbb{R}^{3 \times 3m}$ such that
\begin{equation}
    b_j^{(2)}(P)- b_j^{(2)}(\tilde{P})= \frac{A_j (\alpha- \tilde{\alpha})}{t^2}+ O \big( t^{-\frac{7}{3}+ \epsilon} \big),
\end{equation}
where $\alpha$ is understood as a column vector. Set $A= (A_1^T, \cdots, A_m^T)^T$, where the superscript $T$ represents transposition. Then we can further rewrite \eqref{eq trajectory, modified} as 
\begin{equation} \label{eq trajectory, modified twice}
    \left\{ \begin{aligned}
        \dot{\alpha}(t)- \dot{\tilde{\alpha}}(t) &= 2\beta(t)- 2\tilde{\beta}(t), \\
        \dot{\beta}(t)- \dot{\tilde{\beta}}(t) &= \frac{A (\alpha- \tilde{\alpha})}{t^2}+ O \big( t^{-\frac{7}{3}+ \epsilon} \big), \\
        \dot{\lambda}(t)- \dot{\tilde{\lambda}}_j(t) &= O \big( t^{-2} \big),
    \end{aligned} \right.    
\end{equation}

In other words, we have reduced the problem to the following lemma.

\begin{lem}
Let $0 < \delta< \kappa<1$, $n,m \in \mathbb{N}$, $A \in \mathbb{R}^{n \times n}$, $F \in C^1 \big( \mathbb{R}_+ \times \mathbb{R}^{n+n+m}; \ \mathbb{R}^n \big)$ and $H \in C^1 \big( \mathbb{R}_+ \times \mathbb{R}^{n+n+m}; \ \mathbb{R}^m \big)$. Assume 
\begin{equation} \label{eq uniform C^1 estimate of F}
    \sup_{|\omega| \le 1} \Big( |F(t,\omega)|+ |\nabla_\omega F(t,\omega)| \Big) \le t^{-2-\kappa}, \quad \forall t>0
\end{equation}
and
\begin{equation} \label{eq uniform C^0 estimate of H}
    \sup_{|\omega| \le 1} \Big( |H(t,\omega)|+ |\nabla_\omega H(t,\omega)| \Big) \le t^{-1-\kappa}, \quad \forall t>0.
\end{equation}
Then there exists $T>0$, $x,y \in C^1 \big( [T,+\infty), \mathbb{R}^n \big)$ and $z \in C^1 \big( [T,+\infty), \mathbb{R}^m \big)$ such that
\begin{equation}
    \left\{ \begin{aligned}
        \dot{x}(t) &= y(t) \\
        \dot{y}(t) &= \frac{A x(t)}{t^2}+ F(t, x(t), y(t),z(t)) \\
        \dot{z}(t) &= H(t,x(t),y(t),z(t))
    \end{aligned} \right.
    \quad \text{and} \quad 
    \left\{ \begin{aligned}
        |x(t)| &\le t^{-\delta}, \\
        |y(t)| &\le t^{-1-\delta}, \\
        |z(t)| &\le t^{-\delta}, 
    \end{aligned} \right. \quad \forall t \ge T. 
\end{equation}
\end{lem}

\begin{proof}
We may work instead on $\mathbb{C}$ by setting $F(t,\omega)= F(t, \re(\omega))$, $G(t,\omega)= G(t, \re(\omega))$ for complex $\omega$, and allowing $x(t), y(t), z(t)$ to take complex values. If the complex counterpart is proved, then the lemma follows by taking the real part.

For $T>0$ and $\omega=(x,y,z) \in C \big( [T,+\infty), \mathbb{C}^{n+n+m} \big)$, define the norm 
\begin{equation}
    {\Vert \omega \Vert}_3 := \sup_{t \ge T} \Big( t^{\delta} |x(t)|+ t^{1+\delta} |y(t)|+ t^\delta |z(t)| \Big)
\end{equation}
and let $B= \left\{\omega \in C \big( [T,+\infty), \mathbb{C}^{n+n+m} \big) \ \big| \ \Vert \omega \Vert_3 \le 1 \right\}$. We want to find a solution in $B$.

First we consider the case when $A$ is diagonalizable over $\mathbb{C}$ and $-\frac{1}{4}$ is not an eigenvalue of $A$. We may take $n$ linear independent eigenvectors $v_1, \cdots v_n \in \mathbb{C}^n$ of $A$, with eigenvalues $c_1, \cdots, c_n$, respectively. Let $a_j, b_j$ be the two roots of $\lambda^2-\lambda= c_j$. We have $a_j \neq b_j$ since $c_j \neq -\frac{1}{4}$. Write $F(t,x)= \sum_{j=1}^n f_j(t,x) v_j$. Then $f_j$ also satisfies \eqref{eq uniform C^1 estimate of F}. For $\omega \in B$, we define $\Gamma \omega$ by
\begin{equation} \begin{aligned}
    (\Gamma x)(t) &= \sum_{j=1}^n G_{a_j,b_j} f_j(t,\omega) v_j, \\
    (\Gamma y)(t) &= \frac{\d}{\d t} \sum_{j=1}^n G_{a_j,b_j} f_j(t,\omega) v_j, \\
    (\Gamma z)(t) &= -\int_t^\infty H(\tau, \omega(\tau)) \d \tau, \\
\end{aligned} \end{equation}
where 
\begin{equation}
    G_{a,b} f(t,\omega)= \frac{G_a f(t,\omega)- G_b f(t,\omega)}{a-b}
\end{equation}
and
\begin{equation}
    G_a f(t,\omega)= \left\{ \begin{aligned}
        &\ t^a \int_1^t \tau^{1-a} f(\tau, \omega(\tau)) \d \tau, && \mathrm{Re}(a) \le -\kappa, \\
        &-t^a \int_t^\infty \tau^{1-a} f(\tau, \omega(\tau)) \d \tau, && \mathrm{Re}(a) > -\kappa,
    \end{aligned} \right. 
\end{equation}
for any $a \neq b \in \mathbb{C}$, function $f(\cdot,\cdot)$ satisfying \eqref{eq uniform C^1 estimate of F} and $\omega \in B$. 

Note that, if $f$ satisfies \eqref{eq uniform C^1 estimate of F}, then $G_a f(t,\omega)$ is well-defined and satisfies
\begin{equation}
    |G_a f(t,\omega)| \le Ct^{-\kappa} \log t, \quad |G_a f(t,\omega)- G_a f(t,\omega')| \le Ct^{-\kappa} {\Vert \omega- \omega' \Vert}_3, \quad \forall \omega, \omega' \in B.
\end{equation}
This gives estimates for $\Gamma$ on the $x$ component. Similar estimates hold for the $y$ component. Moreover, \eqref{eq uniform C^0 estimate of H} gives estimates on the $z$ component. These estimates write together as
\begin{equation}
    {\Vert \Gamma \omega \Vert}_3 \le C T^{-\kappa+ \delta} \log T \quad \text{and} \quad {\Vert \Gamma \omega- \Gamma \omega' \Vert}_3 \le CT^{-\kappa+ \delta} {\Vert \omega- \omega' \Vert}_3, \quad \forall \omega, \omega' \in B.
\end{equation}
Therefore, if $T$ is large enough, then $\Gamma$ maps $B$ to itself and is a contraction. 

By direct computation, we have
\begin{equation} \begin{aligned}
    \frac{\d}{\dt} (\Gamma x) &= \Gamma y, \\
    \frac{\d}{\dt} (\Gamma y) &= \frac{A}{t^2} (\Gamma x)+ F(t,x,y,z), \\
    \frac{\d}{\dt} (\Gamma z) &= H(t,x,y,z),
\end{aligned} \end{equation}
thus the unique fixed point of $\Gamma$ in $B$, guaranteed by the contraction mapping theorem, is the desired solution $\omega$.

For the general case, since the set of diagonalizable matrices is dense, for any $c_0>0$, there exists $\tilde{A} \in \mathbb{R}^{n \times n}$ such that $\tilde{A}$ is diagonalizable over $\mathbb{C}$, $\Vert A- \tilde{A} \Vert \le c_0$, and $-\frac{1}{4}$ is not an eigenvalue of $\tilde{A}$. Then we consider the ODE with $A$ replaced by $\tilde{A}$ and $F$ replaced by 
\begin{equation}
    \tilde{F}(t,\omega)= F(t,\omega)+ \frac{(A-\tilde{A}) x}{t^2}.
\end{equation}
Instead of \eqref{eq uniform C^1 estimate of F}, we have
\begin{equation}
    \sup_{\omega \in B} |\tilde{F}(t, \omega(t))| \le c_0 t^{-2-\delta}+ t^{-2-\kappa}, \ \sup_{\omega \in B} |\nabla_\omega \tilde{F}(t, \omega(t))| \le c_0 t^{-2}+ t^{-2-\kappa}, \quad \forall t>0.
\end{equation}
We repeat the construction of $\Gamma$ with $\tilde{A}$ and $\tilde{F}$ and we will get 
\begin{equation}
    {\Vert \Gamma \omega \Vert}_3 \le Cc_0 \quad \text{and} \quad {\Vert \Gamma \omega- \Gamma \omega' \Vert}_3 \le Cc_0 {\Vert \omega- \omega' \Vert}_3, \quad \forall \omega, \omega' \in B
\end{equation}
for some $C>0$. We can still conclude upon taking $c_0$ small enough.
\end{proof}

Back to the proposition, the lemma implies that there exists a solution of \eqref{eq trajectory, modified twice} in $Y$.
\end{proof}

\subsection{Review of the hyperbolic case}

Now, let us go over the proof of the hyperbolic case and see what has to be changed in the other two cases.

Everything in Section \ref{sec approximate} works here, because it does not depend on the dynamics. Proposition \ref{prop hyperbolic trajectory} is replaced by Proposition \ref{prop parabolic and hyperbolic-parabolic trajectory}. The rest of Section \ref{sec reduction} will work because we have only used $a^{(N)}(t) \to \infty$ as $t \to \infty$. Therefore, it suffices to prove Proposition \ref{prop bootstrap} in the other two settings.

In Section \ref{sec estimate}, the asymptotic properties \eqref{eq asymptotic behavior of hyperbolic trajectory} need to be changed. In the parabolic setting, by \eqref{eq parabolic solution} and \eqref{eq estimate of parabolic trajectory}, we have
\begin{equation} \label{eq asymptotic behavior of parabolic trajectory}
    a(t) \sim t^{\frac{2}{3}}, \quad |\alpha_j| \lesssim t^{\frac{2}{3}}, \quad |\beta_j| \lesssim t^{-\frac{1}{3}}, \quad |\dot{\beta}_j| \lesssim t^{-\frac{4}{3}}, \quad \lambda_j \sim 1, \quad |\dot{\lambda}_j| \lesssim t^{-\frac{4}{3}}.
\end{equation}

For hyperbolic-parabolic solutions, the relation on $\{1,2,\cdots,m\}$ given by $a_j=a_k$ is an equivalence relation. Let $M$ denote the set of equivalent classes. For $J \in M$, let $\alpha_J$ be any of $\{\alpha_j|j \in J\}$ and $\beta_J$ be any of $\{\beta_j|j \in J\}$. Then by \eqref{eq hyperbolic-parabolic solution} and \eqref{eq estimate of parabolic trajectory}, we have
\begin{equation} \label{eq asymptotic behavior of hyperbolic-parabolic trajectory} \begin{gathered}
    a(t) \sim t^{\frac{2}{3}}, \quad |\alpha_J| \lesssim t, \quad |\beta_J| \lesssim 1, \quad |\dot{\beta}_J| \lesssim t^{-\frac{4}{3}}, \quad \lambda_J \sim 1, \quad |\dot{\lambda}_J| \lesssim t^{-\frac{4}{3}}, \\
    |\alpha_j- \alpha_J| \lesssim t^{\frac{2}{3}}, \quad |\beta_j- \beta_J| \lesssim t^{-\frac{1}{3}}, \quad |\lambda_j- \lambda_J| \lesssim t^{-\frac{1}{3}}, \quad \forall j \in J. 
\end{gathered} \end{equation}

With these one can check that all the estimates in Section \ref{subsec parameter}, in particular \eqref{eq estimate of Mod with epsilon} and \eqref{eq estimate of Mod}, hold. This is mainly because we did not use the sharp bounds in \eqref{eq asymptotic behavior of hyperbolic trajectory}. 

However, we need some modification in Section \ref{subsec error}. Lemma \ref{lem cutoff} is valid for the parabolic case. For the hyperbolic-parabolic case, we prove the following:

\begin{lem} \label{lem cutoff, mixed case}
There exist $c,C>0$ and $\varphi_j \in C^{1,\infty} (\mathbb{R}_+ \times \mathbb{R}^3)$ for $1 \le j \le m$ such that \eqref{eq cutoff} holds and, moreover, for any $J \in M$,
\begin{equation} \label{eq cutoff, mixed case} \begin{gathered}
    |\partial_t \varphi_J|+ |\nabla \varphi_J| \le Ct^{-1}, \quad \text{where } \varphi_J= \sum_{j \in J} \varphi_j.
\end{gathered} \end{equation}
\end{lem}

\begin{proof}
For $J \in M$, let $\alpha_J$ be any of $\alpha_j$, $j \in J$. Applying Lemma \ref{lem cutoff} to $\{\alpha_J|J \in M\}$, we can find $\varphi_J \in C^{1,\infty} (\mathbb{R}_+ \times \mathbb{R}^3)$ such that 
\begin{equation} \begin{gathered}
    0 \le \varphi_J(t,x) \le 1, \quad \sum_{J \in M} \varphi_J(t,x) \equiv 1, \\
    |\partial_t \varphi_J|+ |\nabla \varphi_J| \le Ct^{-1}, \quad |\partial_t \sqrt{\varphi_J}|+ |\nabla \sqrt{\varphi_J}| \le Ct^{-1}, \\
    \varphi_J(t,x)= \left\{ \begin{aligned}
        &1, \quad |x-\alpha_J(t)| \le ct, \\
        &0, \quad |x-\alpha_K(t)| \le ct,\ K \neq J.
    \end{aligned} \right.
\end{gathered} \end{equation}
Then applying Lemma \ref{lem cutoff} to $\{\alpha_j|j \in J\}$ for each $J \in M$, we find $\psi_j \in C^{1,\infty} (\mathbb{R}_+ \times \mathbb{R}^3)$ such that
\begin{equation} \begin{gathered}
    0 \le \psi_j(t,x) \le 1, \quad \sum_{j \in J} \psi_j(t,x) \equiv 1, \\ 
    |\partial_t \psi_j|+ |\nabla \psi_j| \le Ct^{-\frac{2}{3}}, \quad |\partial_t \sqrt{\psi_j}|+ |\nabla \sqrt{\psi_j}| \le Ct^{-\frac{2}{3}}, \\
    \psi_j(t,x)= \left\{ \begin{aligned}
        &1, \quad |x-\alpha_j(t)| \le ct^{\frac{2}{3}}, \\
        &0, \quad |x-\alpha_k(t)| \le ct^{\frac{2}{3}},\ k \neq j.
    \end{aligned} \right.
\end{gathered} \end{equation}
Finally, take $\varphi_j= \varphi_J \psi_j$, where $J$ contains $j$. Then all the conditions are satisfied.
\end{proof}

It still suffices to prove Proposition \ref{prop coercivity} and \ref{prop estimate on G(epsilon)}. 

We can prove Proposition \ref{prop coercivity} exactly as before. For Proposition \ref{prop estimate on G(epsilon)}, we need to check \eqref{eq estimate of G_1}, \eqref{eq estimate of G_2} and \eqref{eq estimate of G_3}. The proof of \eqref{eq estimate of G_1} need not to be changed. The difficulty of the other two estimates is that $|\partial_t \varphi_j|+ |\nabla \varphi_j|$ does not have an $O(t^{-1})$ decay. By checking the previous computation, we need to show
\begin{equation} \label{eq G_2 cutoff error}
    \sum_{j=1}^m \Big( \frac{1}{\lambda_j^2}+ |\beta_j|^2 \Big) \int \Big( \partial_t \varphi_j |\varepsilon|^2+ 2\nabla \varphi_j \mathrm{Im}(\nabla \varepsilon \overline{\varepsilon}) \Big)= O\left( \frac{\Vert \varepsilon \Vert_{H^1}^2}{t} \right)
\end{equation}
and
\begin{equation} \label{eq G_3 cutoff error} \begin{aligned}
    \sum_{j=1}^m \beta_j \int \bigg( \nabla \varphi_j \Big( 2|\nabla \varepsilon|^2+ 2\phi_{\mathrm{Re} (\varepsilon \overline{R})} \mathrm{Re} (& \varepsilon \overline{R}) + \phi_{|R|^2} |\varepsilon|^2 \Big) \\
    &+ \partial_t \varphi_j \mathrm{Im} (\nabla \varepsilon \overline{\varepsilon}) \bigg) =  O\left( \frac{\Vert \varepsilon \Vert_{H^1}^2}{t} \right).  
\end{aligned} \end{equation}
At this point, we may understand the parabolic case as a special case of the hyperbolic-parabolic case, so we shall focus on the hyperbolic-parabolic case.

Our argument is easier than that in \cite{KMR2bodyHartree}. In fact, it is not clear whether the argument there can be applied here. Using \eqref{eq cutoff} and \eqref{eq asymptotic behavior of hyperbolic-parabolic trajectory}, we have
\begin{equation}
    |\beta_j- \beta_J| \cdot \Big( |\partial_t \varphi_j|+ |\nabla \varphi_j| \Big) \le \frac{C}{t}.
\end{equation}
Combining this and \eqref{eq cutoff, mixed case}, we derive \eqref{eq G_3 cutoff error}. For \eqref{eq G_2 cutoff error}, similarly, if we replace $\lambda_j$ by $\lambda_J$ and $\beta_j$ by $\beta_J$, then the difference is at most $O(t^{-1})$. Finally the terms with $\lambda_J$ or $\beta_J$ are controlled using \eqref{eq cutoff, mixed case}. We remark that this is the only place we need the assumption on the masses.

We have thus completed the proof of the parabolic case and the hyperbolic-parabolic case. Therefore, Theorem \ref{thm existence} is proved.

\section*{Acknowledgments}

I would like to thank Professor Wilhelm Schlag for suggesting the problem and for useful discussions. I would also like to thank Professor Joachim Krieger for kind encouragement.

\bibliographystyle{amsplain}
\bibliography{ref}

\end{document}